\documentclass[authoryear,preprint,3p,times,11pt]{elsarticle}

\usepackage{tikz}
\usetikzlibrary{arrows,intersections}
\usepackage{txfonts}
\usepackage{natbib}
\usepackage{epsfig}
\usepackage{float}
\usepackage{graphicx}
\usepackage{epstopdf}
\usepackage[colorlinks,linkcolor=blue,hyperindex,CJKbookmarks]{hyperref}
\usepackage{array}
\usepackage{latexsym,bm,amsmath,amssymb,amsfonts,amsthm,mathrsfs}
\usepackage{mathrsfs}
\usepackage{color}
\usepackage[toc,page,title,titletoc,header]{appendix}
\usepackage{threeparttable}
\usepackage{lscape}
\usepackage{multirow}
\usepackage{url}
\usepackage{comment}
\usepackage{soul}
\usepackage[mathlines,edtable]{lineno}
\usepackage{lipsum}
\usepackage{cuted}
\usepackage[margins]{misc/trackchanges}
\usepackage{booktabs}
\usepackage{makecell}
\usepackage{graphicx}
\usepackage[utf8]{inputenc}
\usepackage{natbib}
\usepackage{adjustbox}
\usepackage[linesnumbered,ruled,vlined]{algorithm2e}
\usepackage{hyperref}
\usepackage{cleveref} % for enhanced referencin
\usepackage{etoolbox} %% <- for \cspreto, \csappto
\newcommand*\linenomathpatch[1]{%
  \cspreto{#1}{\linenomath}%
  \cspreto{#1*}{\linenomath}%
  \csappto{end#1}{\endlinenomath}%
  \csappto{end#1*}{\endlinenomath}%
}

\linenomathpatch{equation}
\linenomathpatch{gather}
\linenomathpatch{multline}
\linenomathpatch{align}
\linenomathpatch{alignat}
\linenomathpatch{flalign}

\newcommand{\EQ}{\begin{eqnarray}}
\newcommand{\EN}{\end{eqnarray}}
\newcommand{\EQQ}{\begin{eqnarray*}}
\newcommand{\ENN}{\end{eqnarray*}}

\newcommand{\bremark}{\begin{remark} \begin{rm} }
\newcommand{\eremark}{ \end{rm} \rule{1mm}{2mm}\end{remark} }

\newcommand{\basm}{\begin{assumption} \begin{rm}}
\newcommand{\easm}{\end{rm}\end{assumption}}

\newcommand{\btheorem}{\begin{theorem} \begin{rm} }
\newcommand{\etheorem}{ \end{rm} \rule{1mm}{2mm}\end{theorem} }

\newcommand{\blemma}{\begin{lemma} \begin{rm} }
\newcommand{\elemma}{ \end{rm} \rule{1mm}{2mm}\end{lemma} }

\newcommand{\bcorollary}{\medskip\begin{corollary} \begin{rm} }
\newcommand{\ecorollary}{ \end{rm} \rule{1mm}{2mm}\end{corollary} }

\newcommand{\bdefinition}{\medskip\begin{definition}\begin{rm} }
\newcommand{\edefinition}{ \end{rm} \rule{1mm}{2mm}\end{definition} }

\newcommand{\bproposition}{\medskip\begin{proposition} \begin{rm} }
\newcommand{\eproposition}{ \end{rm} \rule{1mm}{2mm}\end{proposition} }

\newcommand{\bexample}{\begin{example} \begin{rm} }
\newcommand{\eexample}{ \end{rm} \rule{1mm}{2mm}\end{example} }

\newcommand{\refEqsRange}[2]{%
    \hyperref[eq#1]{(#1)}--\hyperref[eq#2]{(#2)}
}

\newtheorem{theorem}{\bf Theorem}[section]
\newtheorem{lemma}{\bf Lemma}[section]
\newtheorem{definition}{\bf Definition}[section]
\newtheorem{remark}{\bf Remark}[section]
\newtheorem{corollary}{\bf Corollary}[section]
\newtheorem{proposition}{\bf Proposition}[section]
\newtheorem{example}{\bf Example}[section]
\newtheorem{assumption}{\bf Assumption}[section]
\newtheorem{problem}{\it Problem}

\usepackage{algpseudocode}
\usepackage{algorithmicx}
\usepackage{amssymb}  % Use Input in the format of Algorithm

\usepackage{diagbox}
\usepackage{subcaption}
\usepackage{svg}
\usepackage{color}     % For color support
\usepackage{xcolor}    % For extended color support
\usepackage[T1]{fontenc} % Added to support Polish characters

\graphicspath{{fig/}}
\journal{arXiv}
\begin{document}

\begin{frontmatter}
\title{% An aggregated dynamic approach for large-scale ride-sharing traffic flow networks: modeling and relocation
 Integrated Take-off Management and Trajectory Optimization for Merging Control in Urban Air Mobility Corridors 
% Integrated Take-off Management and Trajectory Optimization for Merging Control in Air Mobility Corridors 
}

\author[sysu]{Yingqi LIU}
\ead{liuyq278@mail2.sysu.edu.cn}
\author[pcl]{Tianlu PAN}
\ead{pantl@pcl.ac.cn}
% \author[ty]{Yazan SAFADI}
\author[sysu]{Jingjun TAN}
\ead{tangjj7@mail2.sysu.edu.cn}
\author[sysu]{Renxin ZHONG\corref{cor}}
\ead{zhrenxin@mail.sysu.edu.cn}
\author[sysu]{Can CHEN\corref{cor}}
\ead{can-caesar.chen@connect.polyu.hk}
\cortext[cor]{Corresponding authors}

\address[sysu]{Guangdong Provincial Key Laboratory
of Intelligent Transportation Systems, School of Intelligent Systems Engineering, Sun Yat-sen University, Shenzhen, China}
\address[pcl]{Department of Network Intelligence, Peng Cheng Laboratory, Shenzhen, China}
% \address[ty]{Technion-Israel Institute of Technology, Israel}
%\address[epfl]{Urban Transport Systems Laboratory, École Polytechnique Fédérale de Lausanne, Switzerland}
\def\corresemail{can.chen@epfl.ch, zhrenxin@mail.sysu.edu.cn}

% \address[sysu]{School of Intelligent Systems Engineering, Sun Yat-sen University, China}
% \address[polyu]{Department of Civil and Environmental Engineering, The Hong Kong Polytechnic University, Hong Kong SAR, China.}
% \address[monash]{Institute of Transport Studies, Department of Civil Engineering, Monash University, Australia}
% \address[cityu]{School of Data Science, City University of Hong Kong, Hong Kong SAR, China.}
% \footnotetext[1]{The authors contributed equally as the first author.}

% \pagewiselinenumbers

\begin{abstract}

Urban Air Mobility (UAM) has the potential to revolutionize daily transportation, offering rapid and efficient aerial mobility services. Take-off and merging phases are critical for air corridor operations, requiring the coordination of take-off aircraft and corridor traffic while ensuring safety and seamless transition. This paper proposes an integrated take-off management and trajectory optimization for merging
control in UAM corridors. We first introduce a novel take-off airspace design. To our knowledge, this paper is one of the first to propose a structured design for take-off airspace. 
Based on the take-off airspace design, we devise a hierarchical coordinated take-off and merging management (HCTMM) strategy.
To be specific, the take-off airspace design can simplify aircraft dynamics and thus reduce the dimensionality of the trajectory optimization problem whilst mitigating obstacle avoidance complexities. 
The HCTMM strategy strictly ensures safety and improves the efficiency of take-off and merging operations.
At the tactical level, a scheduling algorithm coordinates aircraft take-off times and selects dynamic merging points to reduce conflicts and ensure smooth take-off and merging processes. At the operational level, a trajectory optimization strategy ensures that each aircraft reaches the dynamic merging point efficiently while satisfying safety constraints.
%Simulation results show a significant reduction in flight time and control cost compared to the baseline method while maintaining safety across various corridor traffic conditions. 
Simulation results show that, compared to representative strategies with fixed or dynamic merging points, the HCTMM strategy significantly improves operational efficiency and reduces computational burden, while ensuring safety under various corridor traffic conditions. Further results confirm the scalability of the HCTMM strategy and the computational efficiency enabled by the proposed take-off airspace design. These findings indicate that the proposed method provides a flexible framework for future UAM operations, supporting the smooth merging of electric vertical take-off and landing (eVTOL) aircraft into UAM corridors.

\end{abstract}

\end{frontmatter}

% \pagewiselinenumbers

\section{Introduction}
\label{sec:Introduction}

Urban transportation is facing unprecedented challenges, with escalating congestion, pollution, and infrastructure strain posing significant concerns. As a groundbreaking mobility solution, Urban Air Mobility (UAM) has the potential to redefine the future of transportation. By leveraging low-altitude airspace, UAM can overcome the limitations of the ground transportation system and unlock an entirely new three-dimensional (3D) urban mobility paradigm. Advances in machine intelligence, vertical take-off and landing technologies, and battery innovations have provided greener, quieter, and more autonomous air vehicles, specifically electric vertical take-off and landing (eVTOL) aircraft. Moreover, the concept of UAM promises effective, automated, and environmentally friendly passenger and cargo air transport services \citep{Bradford2020UAM}. Given these advantages, the development of UAM has emerged as a central priority in global urban planning. \cite{mayakonda2020top} estimated that between 327 million and 1,207 million passenger trips will occur in 31 cities around the world with a potential deployment for UAM by 2035.  As one of the pioneers, China has developed a UAM market valued at 500 billion RMB by 2023 and is projected to exceed two trillion RMB by 2030 \citep{xinhuabao2024}.

As large-scale deployment of UAM approaches, comprehensive airspace organization and management are essential to ensure the safety and efficiency of UAM systems. The current literature mainly proposes four types of low-altitude airspace structures: full-mixed \citep{sunil2015metropolis}, layer-based \citep{wu2021trajectory}, zone-based \citep{weng2025urban}, and corridor-based \citep{el2024fixed}. Among them, the corridor-based structure utilizes designated UAM corridors to separate UAM operations from other air traffic \citep{bauranov2021designing}. In addition, it provides a fixed routing strategy that ensures safe and separated eVTOL traffic, increases throughput, and reduces the number of potential conflicts \citep{el2024fixed,bauranov2021designing}. Owing to its well-organized traffic regulations, enhanced safety measures, and ability to foster public trust, the corridor-based structure is widely regarded as the leading approach for low-altitude airspace organization in the short to medium term \citep{cohen2021urban}. Accordingly, an increasing number of studies have begun to focus on the design of corridor-based airspace, mainly addressing aspects such as the cruise corridor network layout and its detailed configuration. To address airspace constraints caused by complex urban environments and existing conventional air traffic systems, \cite{el2024fixed,slama2022generating} designed customized safe cruise corridor networks for densely built urban and airport areas. Building on the safe corridor networks, \cite{Fontaine2023UAM} explores variations in UAM corridor topology designed to address specific challenges, such as the implementation of ``passing zones''. \cite{he2022route} defines the cruise corridor structure design for corridor-based airspace, including the geometric parameters of unidirectional paths, buffer zone configurations, and safety separation requirements.
% Further research has concentrated on measuring corridor capacity \citep{muna2021air} and investigating how it impacts cruise route planning. 
However, most existing literature primarily focuses on the cruise phase, while airspace utilization policies and design specific to the take-off and merging phases are still in their infancy. 

The take-off and merging phases of an eVTOL start with a scheduled departure from the vertiport and end upon reaching the designated corridor and merging at a prescribed merge point.
As UAM systems evolve, there will be a dense distribution of vertiports, an increasingly complex corridor network, and a significant rise in take-off demand \citep{Fontaine2023UAM}. In this context, the absence of well-defined airspace design poses significant challenges for eVTOLs in cooperative trajectory optimization and obstacle avoidance during the take-off and merging phases. These challenges highlight the importance of a comprehensive airspace design for take-off and merging operations.

%然后在这里写现在关于这俩阶段的管理的文献分为两类
In addition to well-defined airspace design, the management strategies in the take-off and merging phases are equally critical, as they directly affect airspace capacity and affect flight safety \citep{doole2022investigation}.
According to the simulation results by \cite{cummings2023measuring}, weaving areas with traffic merging are the main bottlenecks with higher conflict rates and lower traffic speed of air corridors. In light of these traffic challenges, most studies focus on scenarios involving a single corridor-vertiport pair take-off and merging management and can be categorized into two classes: tactical conflict management strategies and operational trajectory optimization strategies.

%然后上层要突出动态合并点选择的缺失，主要是针对静态合并点
Tactical conflict management involves decisions across spatial and temporal dimensions to balance traffic demand with airspace capacity at bottlenecks, such as take-off times and merging points.
The Demand Capacity Balancing (DCB) approach uses the capacity-constrained bottleneck model to optimize the take-off times of eVTOLs. \cite{chen2024integrated} developed an optimization-based DCB algorithm, which takes the scheduled take-off time of all aircraft as input and calculates the optimal times required to minimize total delay in advance. Considering extreme scenarios where take-off demand exceeds vertiport capacity, \cite{lee2022demand} validated that the proposed DCB algorithm could work well to manage the demand effectively by assigning pre-departure schedule delays. 
As for merging points, on-ramp structures in ground traffic have also been applied to  UAM operations by simplifying the merging process of eVTOLs into one or more fixed merging points. 
Building on this concept, \cite{ren2023aircraft} proposed a scheduling model for flight sequences entering the corridors via ``the on-ramp'' that minimizes total delay. \cite{LIANG2018207} considered multiple fixed merging points and established a Multi-Level Point Merge (ML-PM) model for departing runway allocation and take-off time optimization, which enables more efficient and realistic conflict-free take-off management. Intuitively, merging into air corridors would have more flexibility than that of ground traffic since vehicular trajectories are restricted by the topology of on-ramps (thus fixed merging points).  Relying solely on fixed parameters such as fixed merging points and constant take-off and merging durations may reduce the flexibility and efficiency of UAM systems, especially when the take-off demand and corridor traffic are heavy. 
% To address this limitation, this paper replaces fixed merging points by selecting dynamic merging points through real-time coordination with the air corridor traffic.
To address this limitation, this paper develops a dynamic merging point selection method via real-time coordination with the air corridor traffic.
To the best of our knowledge, this is the first work to introduce dynamic merging points for corridor-based take-off and merging operations in UAM. 
Once the dynamic merging points are adopted, it becomes necessary to optimize take-off time to synchronize the take-off and merging processes with the corridor traffic state, rather than assuming a constant duration as in \cite{lee2022demand,chen2024integrated}\footnote{This constant duration assumption presents an inherent trade-off: excessively conservative durations may reduce scheduling efficiency, while aggressive estimates may compromise safety requirements.}.

% However, as the take-off demand and corridor traffic increase, dynamic merging points selected through real-time coordination with the main traffic flow are necessary. This also implies that, in optimizing take-off schedules and merging point selection, the take-off and merging processes should not be overlooked or simplistically treated as requiring a constant duration.
% In this manner, a conservative constant may result in low take-off schedule efficiency, while an aggressive constant may cause take-off time strategies to fail in meeting safety requirements.

% Although existing studies provide valuable insights into tactical management, most approaches assume fixed merging points and neglect the dynamic nature of take-off and merging phases. This limitation restricts their applicability in scenarios with rapidly increasing take-off demand and corridor traffic. Therefore, to accommodate the evolving trends in future air traffic management, dynamic merging points selected through real-time coordination with the main traffic flow are necessary.
% most existing studies overlook the take-off and merging processes but treat the time required for eVTOLs to complete these operations as a constant.
% In this manner, a conservative constant may result in low take-off schedule efficiency, while an aggressive constant may cause take-off time strategies to fail in meeting safety requirements.

%看看要不要加一点新找的文献。研究一下要不要把强化学习的东西放进来
Operational trajectory optimization strategies involve optimizing the speed and/or acceleration of eVTOLs, aiming to generate safe and efficient take-off and merging trajectories \citep{doole2022investigation, preis2023time}.
Existing studies on take-off and merging trajectory optimization can be broadly categorized into two main approaches: reinforcement learning (RL)-based methods and optimal control problems (OCP). RL-based methods have been applied to coordination problems of take-off and merging due to their scalability and robustness to environmental disturbances. \cite{deniz2024reinforcement} proposed a multi-agent RL (MARL) framework to coordinate the longitudinal speed profiles of up to 20 eVTOLs passing through a fixed merging point.
However, adopting more flexible dynamic merging points and operating in more permissive airspace impose numerous hard constraints on trajectories, such as collision avoidance \citep{wu2022convex}, mechanical limitations \citep{wu2022convex,liu2024flight}, and airspace boundary compliance \citep{park2023trajectory,lu2023multi,wu2024convex}. Existing RL-based take-off and merging control can only indirectly consider constraints by incorporating safety-related penalties into the reward function, which can hardly provide safety guarantees \citep{zhou2024enhancing}.
In contrast, OCP-based approaches explicitly model system dynamics and constraints, providing deterministic solutions that strictly adhere to safety requirements.
\cite{wu2022convex} formulated an OCP for eVTOL merging coordination, considering collision avoidance and mechanical limitations. Similarly, \cite{park2023trajectory} optimized energy-efficient 2D take-off and merging trajectories with path and vortex ring state avoidance constraints within a uniform wind.
Nevertheless, most existing OCP-based methods focus on light corridor traffic.
In scenarios with heavy take-off demand and corridor traffic, %relying solely on trajectory optimization 
the existing OCP approaches may require eVTOLs to detour for obstacle avoidance or hover near the corridor waiting for a suitable merging opportunity. This significantly reduces efficiency and increases energy consumption.

To conclude, integrating tactical conflict management and operational trajectory optimization strategies can enhance the safety and efficiency of UAM systems.
The concept of integrating tactical and operational strategies has been studied in the context of unmanned aerial vehicle (UAV) formation aggregation and reconfiguration. At the tactical level, heuristic algorithms \citep{gao2023hybrid} and integer programming (IP) approaches \citep{wang2021formation,zhang2025hierarchical} were employed to determine the geometric structure and spatial position of the target formation. At the operational level, distributed control algorithms, such as model predictive control (MPC) \citep{wang2021formation,gao2023hybrid}, were employed to guide each UAV to reach its target position safely. Similar control methods have been adopted in automated guided vehicles (AGVs) \citep{yang2018integrated,yue2022dynamic} and flexible job shop (FJS) scheduling \citep{li2022bilevel,saouabi2024two,jiang2025bi}. In both domains, the tactical level is responsible for assignment decisions and global resource allocation, focusing on macroscopic-level planning. The operational level undertakes path planning and job sequencing, emphasizing detailed execution. These hierarchical scheduling problems can be solved by bilevel RL algorithms \citep{li2022bilevel}, heuristic algorithms \citep{yang2018integrated,saouabi2024two,jiang2025bi}, or their combinations \citep{yue2022dynamic}.
The on-ramp merging trajectory optimization problem for connected and automated vehicles (CAVs) is most closely aligned with the merging problem investigated in this paper. Recent studies solve the optimal merging sequence problem for on-ramp CAVs by the mixed integer programming (MIP) \citep{jing2019cooperative,mu2021event,tang2022novel,chen2023integrated}. However, the MIP approach may become computationally intractable for large-scale problems, which motivates the development of more computationally efficient rule-based alternatives.
The rule-based methods include safety-based methods, virtual mapping, etc. To minimize travel time and ensure safe merging, \cite{xue2022platoon} defined the safe-valid gap and determined the pre-target merging gap selection principle. According to \cite{duret2020hierarchical}, the virtual mapping method compares arrival times to a prescribed merging point for potentially conflicting CAVs.
Vehicles that reach the merging point first will be assigned a higher priority. 
% Based on the sequence and merging point assigned, the operational trajectory optimization models in different studies target different objectives favoring traffic efficiency \citep{jing2019cooperative,shi2023cooperative}, energy consumption \citep{duret2020hierarchical,xue2022platoon}, and passenger comfort level \citep{jing2019cooperative,jing2022integrated} while subject to vehicle dynamics, safety requirements and technical constraints. 
Based on the merging sequence, the operational trajectory optimization models in different studies target different objectives, such as traffic efficiency \citep{jing2019cooperative,shi2023cooperative}, energy consumption \citep{duret2020hierarchical,xue2022platoon}, and passenger comfort level \citep{jing2019cooperative,jing2022integrated}. 
These models are subject to vehicle dynamics, safety requirements, and technical constraints.
% Compared to the first-in-first-out sequence and no-control methods, \cite{jing2019cooperative} demonstrated that their two-level framework reduces fuel consumption and travel time while enhancing passenger comfort. 
% \cite{duret2020hierarchical} showed that their approach can avoid unsafe over-reactive deceleration/acceleration maneuvers and improve outflows at network discontinuities.
The aforementioned strategies from UAV formation, AGV and FJS scheduling, and ground transportation provide valuable insights for studying the take-off and merging problems in UAM corridors. However, unlike the fixed target assignments in UAV formation, the merging points in the tactical level for eVTOL are inherently dynamic and time-varying. Meanwhile, at the operational level, the trajectory planning for eVTOLs is a dynamic and continuous state space problem, unlike the mostly static and discrete state space planning in AGV and FJS scheduling. In addition, the trajectory planning for eVTOLs involves higher-dimensional space and must address more mechanical and safety constraints compared with those encountered in the merging control of CAVs or the less constrained UAV formation scenarios. Due to these difficulties, existing methods (from other domains) cannot be directly applied in the take-off and merging control of eVTOLs to UAM corridors.   %, requiring customized modeling and coordination approaches.}

% For corridor-based airspace structures, the design and coordination in the take-off and merging phases are crucial factors that limit airspace capacity and impact flight safety. 
% The current literature lacks a well-defined concept for the design of take-off airspace, hindering the development of safe and efficient coordination mechanisms. 
% Besides, current management strategies in the take-off and merging phases focus primarily on tactical conflict management or operational trajectory optimization. However, tactical conflict management strategies may result in low take-off schedule efficiency or struggle to guarantee safety requirements.
% Operational trajectory optimization strategies may be constrained by unnecessary detours or airborne holding caused by localized obstacle avoidance in scenarios with high take-off demand and corridor traffic.
%1.要不要和摘要一样强调是第一个空域设计？2，要不要突出是第一次采用动态合并点？审稿人二说：. If such methods exist, the authors should cite them to better contextualize their contribution. Conversely, if this is the first work to adopt such a mechanism, the authors are encouraged to state this explicitly and support the claim with a brief discussion of related literature.

To address the aforementioned limitations of existing corridor-based airspace design and management strategies, we develop an integrated strategy to enhance mobility and
flight safety in take-off and merging operations.
We propose a structured design for the take-off airspace to reduce the dimensionality of the trajectory optimization problem and alleviate obstacle avoidance challenges during peak demand periods.
Furthermore, we develop a hierarchical coordinated take-off and merging management (HCTMM) strategy.
At the tactical level, we develop an algorithm to coordinate take-off timing and select dynamic merging points for each merging eVTOL. This algorithm mitigates potential conflicts by coordinating vertiport take-off demands across temporal and spatial dimensions.
At the operational level, we propose a trajectory optimization strategy to jointly minimize time and control cost by optimizing the acceleration profile and trajectory of each merging eVTOL. Additionally, the optimal taking-off and merging trajectory is subject to safety constraints and boundary conditions provided by the tactical level. 
To validate the effectiveness of the airspace design and HCTMM strategy, we consider both single and multiple corridor-vertiport scenarios. We validate the proposed approach on a single corridor-vertiport pair to demonstrate reliable safety assurances across different flow levels. Compared with representative strategies employing fixed or dynamic merging points, the HCTMM strategy achieves markedly improved travel efficiency (including shorter flight times and lower control costs) and a significantly reduced computational burden.
 Furthermore, simulation results on multiple corridor-vertiport pairs validate the scalability of the HCTMM strategy and the effectiveness of the airspace design in improving computational efficiency.

The remainder of this paper is organized as follows: 
\autoref{Problem statement} introduces the take-off airspace design and the problem formulation. \autoref{sec: Tactical conflict management} introduces tactical level management, including dynamic merging point selection and take-off time coordination.
\autoref{sec: Operational deconfliction control} devises the operational level management scheme to generate control and a time-optimal take-off and merging trajectory.  
\autoref{sec: Numerical Experiments} presents the simulations to verify the advancement of the proposed take-off airspace design and HCTMM strategy. \autoref{sec:Conclusion} concludes the paper. The appendices summarize key notations and present companion materials. % and proofs. 

\section{Take-off airspace design and take-off, merging control problem formulation}
\label{Problem statement}
In this section, we present the design principles of the take-off airspace and describe the take-off and merging problem. First, in \autoref{sec: Take-off airspace design}, we introduce the structured take-off airspace design and the simplification of vehicle dynamics. Next, in \autoref{sec: multi-layer corridor structure}, we propose a multi-layer corridor structure and a set of ``direct-indirect'' accessibility rules to support the deployment of structured take-off airspace design, particularly for more complex corridor networks. Then, in \autoref{sec:Trajectory optimization for take-off and merging control}, based on the proposed structured take-off airspace, we formulate the take-off and merging control problem. 

\subsection{Take-off airspace design}
\label{sec: Take-off airspace design}
According to the regulations of the European Union Aviation Safety Agency (EASA) \citep{EasyAccess2024}, the take-off procedure of an eVTOL operation is generally divided into two sequential phases: a vertical take-off phase and a take-off climb phase. A transition point $D$ is defined at an altitude $L_d$ above the vertiport, which marks the end of the vertical take-off segment and the beginning of the climb phase. During the vertical take-off phase, the eVTOL exhibits motion only in the vertical direction. In this phase, the vertiport enforces two regulations: maintaining an obstacle-free volume surrounding the vertiport and imposing a minimum take-off interval \citep{song2022optimal, EasyAccess2024}. These measures establish a well-defined structure and stringent controls that guarantee safe and predictable operations.
In contrast, once the eVTOL enters the take-off climb phase, it performs a three-dimensional maneuver involving vertical ascent, with forward acceleration and potential lateral adjustments required for merging into the designated corridor. In this phase, each merging eVTOL $i \in \mathcal{I}$ operates with full degrees of freedom in position and speed in the unstructured airspace, where $\mathcal{I}$ denotes the set of merging eVTOLs planned for take-off in the region. The dynamics of eVTOL is nonlinear and involves 4-dimensional control variables, i.e., the net thrust $F_i$ and Euler angles $(\alpha_i, \beta_i, \gamma_i)$, and 6-dimensional state variables, i.e., the position $ (x_i, y_i, z_i) $, and the speed $ (v_{i}^x, v_{i}^y, v_{i}^z) $. Following \cite{quan2017introduction}, we present the detailed dynamics in \autoref{sec: Appendix}. 
However, high-dimensional state and control spaces present significant challenges for real-time collaborative trajectory optimization. 
Nevertheless, the high degrees of freedom in trajectories not only amplify the risk of potential conflicts but also significantly increase the computational complexity of conflict resolution strategies. These make it more difficult to plan safe and efficient take-off and merging trajectories for eVTOLs, particularly during peak demand periods.

%above the obstacle limitation surface (OLS)
\begin{figure}[h]
    \centering
    \includegraphics[width=1.0\textwidth]{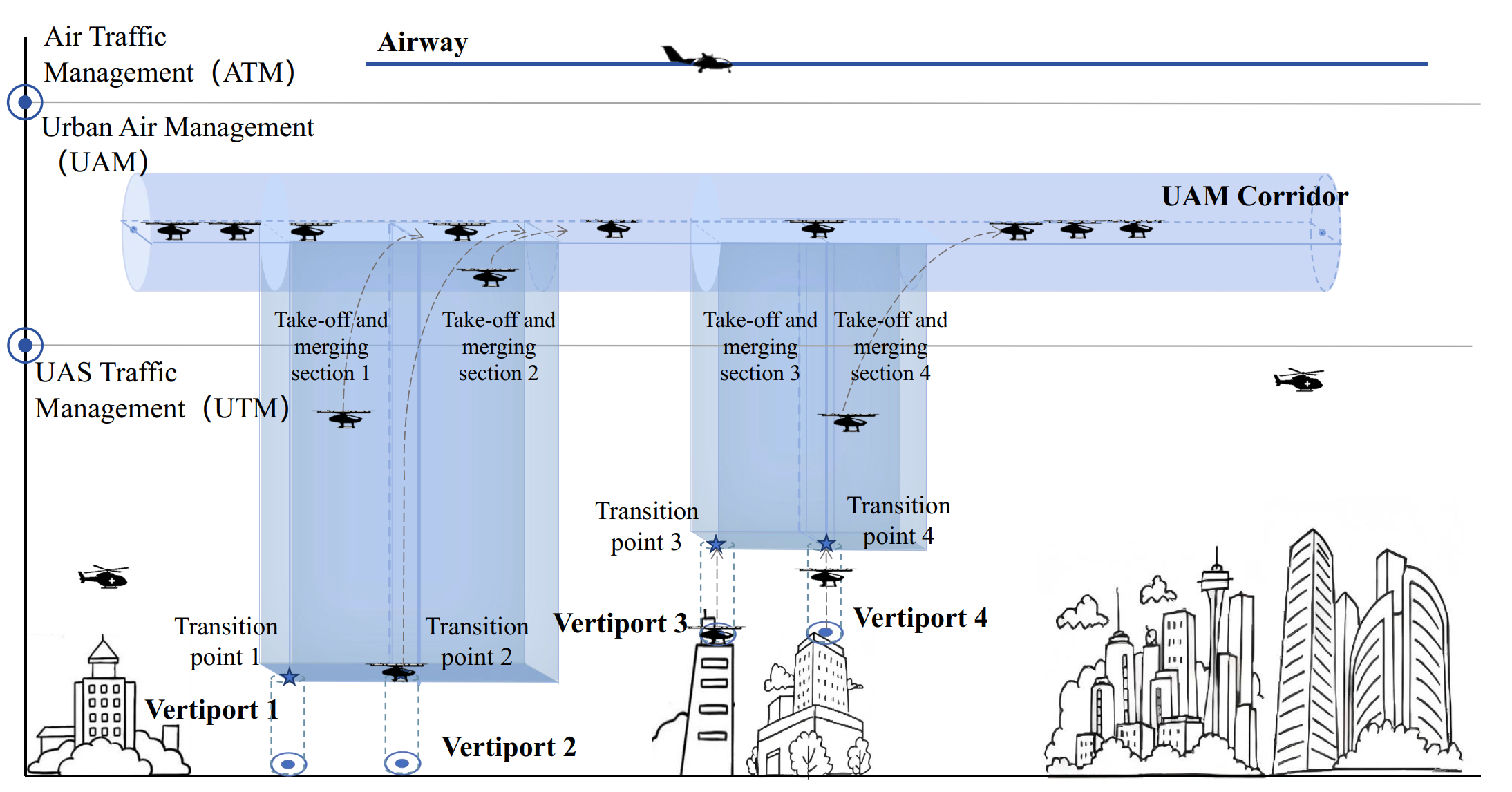}
    \caption{Design concept for take-off airspace with separated TM sections.}
    \label{fig:05}
\end{figure}

To address these challenges, we separate a dedicated Take-off and Merging (TM) section for each pair of vertiport and the corresponding corridor, as illustrated in \hyperref[fig:05]{Figure \ref*{fig:05}}. The design of these TM sections follows a well-defined geometric structure. Specifically, each TM section is defined as the intersection volume between a quadrangular prism and the airspace above the obstacle limitation surface (OLS) of the vertiport. The quadrangular prism is constructed with parallelogram sides and two equal-sized rectangular bases, truncated by the transition point plane and the central plane of the corridor. Notably, the length of the quadrangular prism corresponds to the preset maximum merging length, while the width is equal to the diameter of the UAM corridor. 
In this paper, we focus on the ``single-lane'' UAM corridor configuration, which is widely adopted in existing structured air traffic take-off and merging studies \citep{wu2022convex,ren2023aircraft,deniz2024reinforcement}. Under this assumption, lateral maneuvers are restricted, making the width of the TM section a secondary factor in trajectory planning. 
Consequently, we utilize the TM section structure to confine the take-off climb and merging trajectories of each corridor-vertiport pair to a two-dimensional (2D) plane. 

Building on the structured take-off airspace, we perform a series of coordinate transformations to simplify the dynamics. We first translate the inertial coordinate system $X_{i}Y_{i}Z_{i}$ for each eVTOL $i$ by shifting its origin from the vertiport $O_i$ to the corresponding transition point $D_i$. The new coordinate system for eVTOL $i$ is denoted as $X_i^* Y_i^* Z_i^*$ and is defined by:

\begin{equation}
\label{eq:translation_i}
\begin{bmatrix}
x_i^* \\ y_i^* \\ z_i^*
\end{bmatrix}
=
\begin{bmatrix}
x_i \\ y_i \\ z_i
\end{bmatrix}
-
\begin{bmatrix}
0 \\ 0 \\ L_d
\end{bmatrix},
\quad
O_i = (O_i^x, O_i^y, O_i^z), \quad D_i = (O_i^x, O_i^y, O_i^z + L_d).
\end{equation}
After this coordinate transformation, we constrain the take-off climb trajectory of each eVTOL $i$ to the $X_i'D_iZ_i'$ plane, which is defined by the $X_i'$-axis, the $Z_i'$-axis and the new origin $D_i$ of the $X_i' Y_i' Z_i'$ coordinate system. The $X_i' Y_i' Z_i'$ system is obtained by applying a clockwise rotation of angle $\phi_i \in (-90^\circ,90^\circ)$ about the $X_i^*$-axis to the translated system $X_i^* Y_i^* Z_i^*$, with the associated rotation matrix $R_{2,i}$:

\begin{equation}
\label{eq:rotation_i}
\begin{bmatrix}
x_i' \\ y_i' \\ z_i'
\end{bmatrix}
= R_{2,i}
\begin{bmatrix}
x_i^* \\ y_i^* \\ z_i^*
\end{bmatrix}, 
\quad
R_{2,i} =
\begin{bmatrix}
1 & 0 & 0 \\
0 & \cos\phi_i & -\sin\phi_i \\
0 & \sin\phi_i & \cos\phi_i
\end{bmatrix}.
\end{equation}
Then we introduce the tip-path-plane pitch angle $\theta_i$ to represent the angle between the thrust vector and the $ Z_i'$-axis within the $X'_iD_iZ_i'$ plane, thereby replacing the use of attitude angles $(\alpha_i, \beta_i, \gamma_i)$ in describing the eVTOL's dynamics.
Based on these, we further simplify the system dynamics, reducing the dimension from 10 to 6, with two control variables ($F_i$ and $\theta_i$) and four state variables $(x'_i, z'_i, v_i^{x'}, v_i^{z'})$. 
This can simplify the subsequent control design and significantly reduce the computational burden when solving the trajectory optimization problem. Details are presented in \autoref{sec: Appendix}. 
Moreover, these TM sections are designed to be non-overlapping, ensuring spatial separation of take-off climb and merging operations across different corridor-vertiport pairs. This spatial decoupling confines trajectory conflict management to within individual corridor–vertiport pairs and reduces the number of safety constraints that need to be considered. As a result, the large-scale cooperative optimization problem can be decomposed into a set of single-eVTOL trajectory planning subproblems, which significantly improves computational efficiency and enhances the feasibility of real-time implementation.

% Moreover, these TM sections are designed to be non-overlapping, ensuring spatial separation of take-off and merging operations across different corridor-vertiport pairs. Such a design minimizes the frequency of obstacle avoidance maneuvers for each eVTOL during peak take-off periods. As a result, the number of safety constraints that need to be considered during take-off and merging management is also reduced. This further increases computational efficiency, leading to faster trajectory optimization and improved feasibility for real-time implementation.

\subsection{Multi-layer corridor structure}
\label{sec: multi-layer corridor structure}

As the corridor networks above vertiports become increasingly complex, separating TM section for each corridor-vertiport pair becomes increasingly difficult. In this paper, we design a multi-layer corridor structure and a set of ``direct-indirect'' accessibility rules to support different TM sections' separation.

\begin{figure}[!h]
    \centering
    \begin{subfigure}[c]{0.45\textwidth}
        \centering    
        \includegraphics[width=\textwidth]{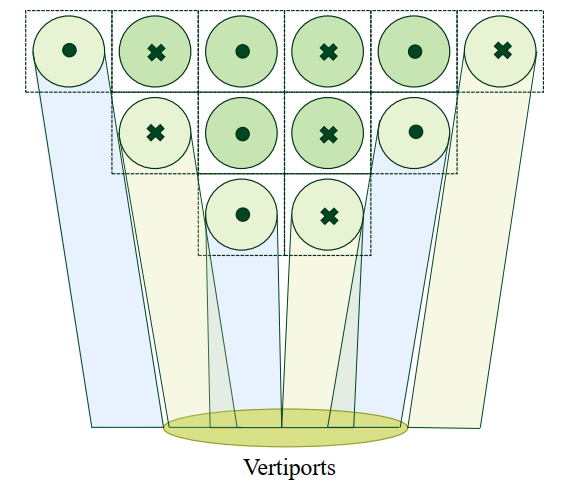}
        \caption{Cross-sectional view of the multi-layer corridor structure over the vertiport}
        \label{fig:03}
    \end{subfigure}
    \hfill
    \begin{subfigure}[c]{0.475\textwidth}
        \centering
        \includegraphics[width=\textwidth]{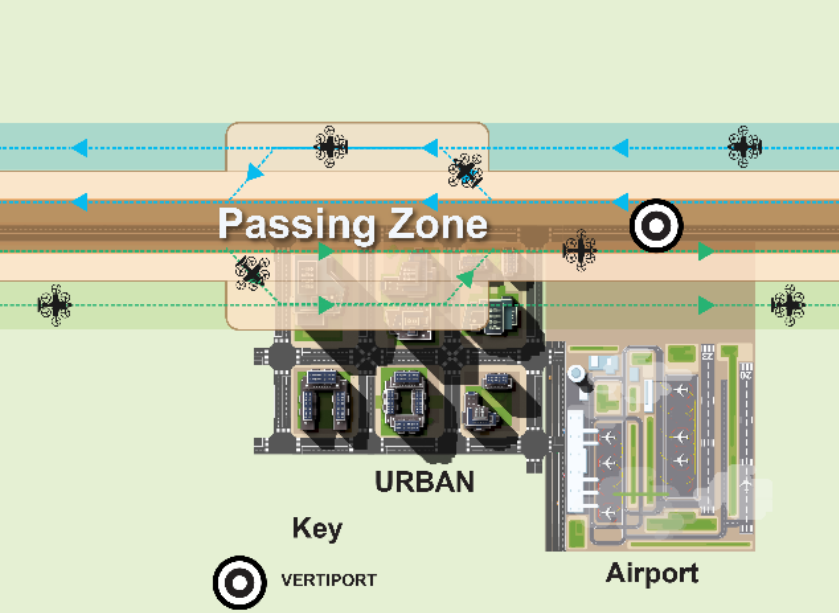}
        \caption{Passing zones schematic from FAA \citep{Fontaine2023UAM} }
        \label{fig:04}
    \end{subfigure}
    \caption{Multi-layer corridor structure design.}
\end{figure}

The multi-layer corridor structure is roughly distributed as a semicircle, where the outer ring is mainly for short-range flights, while the inner ring is for long-range flights. Moreover, the inbound and outbound corridors are alternately distributed.
In the cross-sectional view \hyperref[fig:03]{Figure \ref{fig:03}}, the symbol ``\textbf{·}'' denotes inbound corridors, while ``\textbf{$\times$}'' represents those outbound corridors. Inspired by the concept of passing zones that facilitate lane changes between corridors \citep{Fontaine2023UAM} (see \hyperref[fig:04]{Figure \ref*{fig:04}}), we develop a set of ``direct-indirect'' connectivity rules. In the multi-layer structure, the outer ring corridors are directly connected to vertiports, allowing eVTOLs to merge directly from vertiports into these corridors. Conversely, the inner corridors are indirectly connected, requiring eVTOLs to first merge into the outer corridors before transitioning to the inner corridors via passing zones. It is important to note that due to the size of vertiports, not all outer corridors can be directly connected, particularly those in the upper layers. 
With this multi-layer corridor structure and the ``direct-indirect" connectivity rules, we can accommodate complex corridor traffic flow and separate the TM section for each corridor-vertiport pair (see, e.g., \hyperref[fig:03]{Figure \ref*{fig:03}}).

\subsection{Trajectory optimization for take-off and merging control}
\label{sec:Trajectory optimization for take-off and merging control}
As illustrated in \hyperref[fig:05]{Figure~\ref*{fig:05}}, the eVTOL first accelerates vertically to the take-off safety speed $V_{TOSS}$, and then ascends at a constant vertical speed until reaching the transition point during the vertical take-off phase. The take-off safety speed $V_{TOSS}$ is aligned with the $Z'$-axis, and its magnitude, denoted by $v_{TOSS}$. Then, each eVTOL \(i\) performs its take-off climb and merging trajectories within an \((X'_i D_i Z_i')\) plane defined by the structured take-off airspace design. \hyperref[fig:07]{Figure \ref*{fig:07}} depicts the 2D coordinate system for take-off climb and merging processes. The origin $(D_i)$, corresponding to the transition point of vertiport $(O_i)$, is attached to the transition point, with the horizontal $X'_i$-axis aligned with the direction of the corridor and the vertical $Z_i'$-axis perpendicular to the corridor direction.
Inspired by \cite{xue2022platoon,shi2023cooperative}, we establish an observation zone upstream of the vertiport. The length of the observation zone $L_{o}$ depends on the communication range of the vertiport infrastructure. 
For clarity, we define eVTOLs taking-off from vertiports as ``merging eVTOLs'' and those entering the observation zone as ``corridor eVTOLs''. 
Existing literature on UAM management typically emphasizes prioritizing airborne aircraft over those awaiting take-off to maintain orderly airspace operations \citep{sacharny2022lane,chen2024integrated}, with the National Aeronautics and Space Administration (NASA) management framework serving as a prominent example \citep{kopardekar2016unmanned}. In this paper, we adopt a similar principle by safeguarding the priority right-of-way for corridor eVTOLs to ensure the stability and controllability of corridor operations.

% For instance, the National Aeronautics and Space Administration's (NASA) UTM framework incorporates a strategic deconfliction mechanism, which assigns higher priority to drones already in flight, while takeoff requests are scheduled based on airspace availability \citep{kopardekar2016unmanned}. 

\begin{figure}[h]
    \centering
    \includegraphics[width=1.0\textwidth]{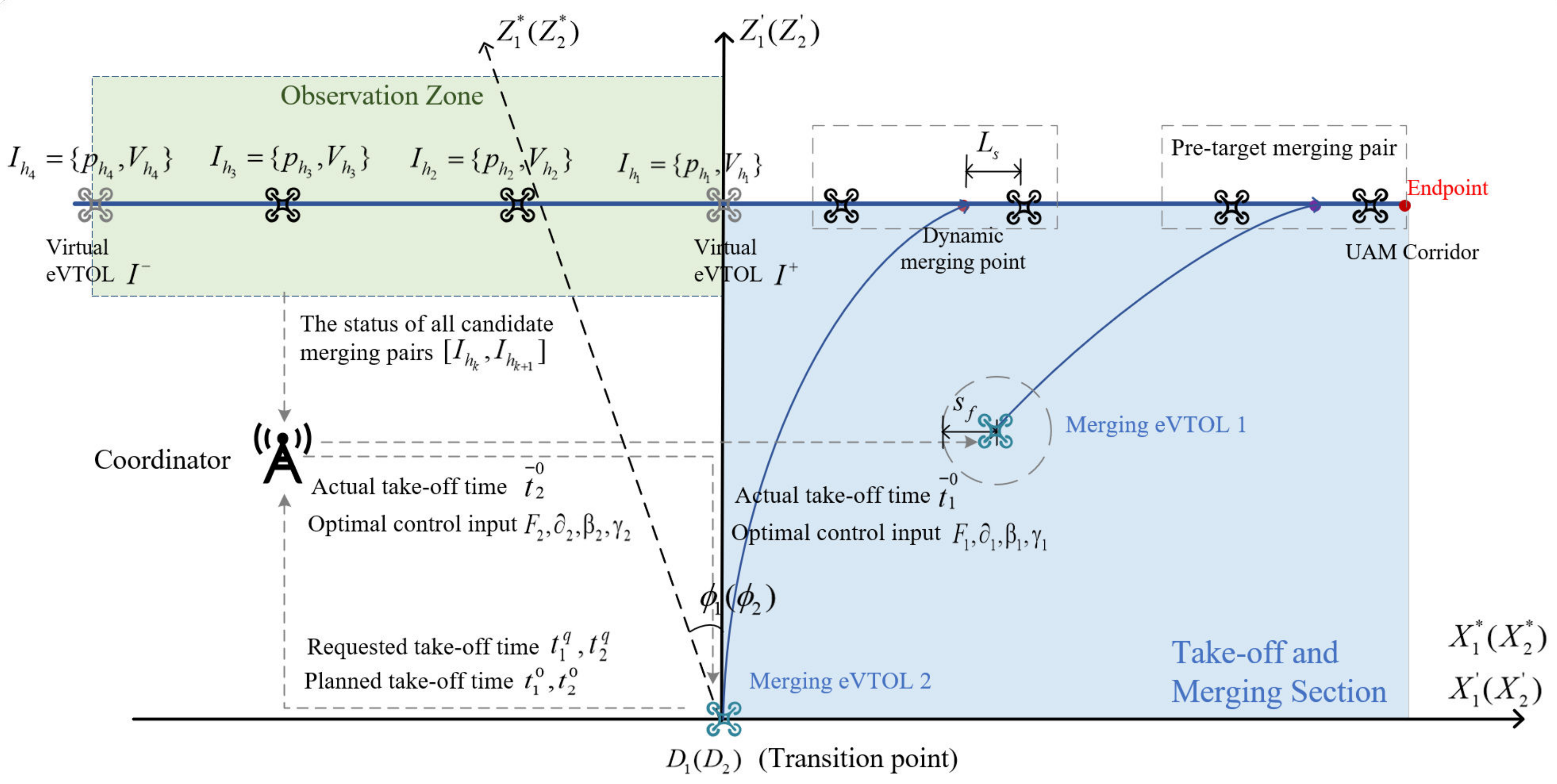}
    \caption{Illustration of eVTOLs take-off climb and merging scenario.}
    \label{fig:07}
\end{figure}

In the corridor, the state $I_{h_k}(t)=\{p_{h_k}(t), V_{h_k}(t)\} $ of eVTOL $h_k$ can be accessed by the central coordinator in real time, where
the position vector $p_{h_k}(t)=[x'_{h_k}(t),z'_{h_k}(t)] $ and speed vector $V_{h_k}(t)=[v_{h_k}^{x'}(t),v_{h_k}^{z'}(t)] $ represent the spatial coordinates and the corresponding speed components in the horizontal and vertical axis, respectively.  
Let $H(t)$ denote the set of corridor eVTOLs and $N_{H}(t)$ denote the number of corridor eVTOL sets at time $t$. Using the horizontal position $x'_{h_k}(t)$ and the vertical position $z'_{h_k}(t)$ of eVTOL $h_k$, %relative to a reference two-dimensional coordinate system, 
we can define $H(t)$ as:
\begin{equation}\label{eq1}
H(t) = \left \{ h_k: h_k\in K(t), -L_{o}\le x'_{h_k}(t)\le 0, z'_{h_k}(t)=(L_c-L_d) \sec\phi_i \right \} 
\end{equation}
where $K(t)$ denotes the set of all eVTOLs within the corridor at time $t$ and  $L_{c}$ is the height of the corridor. To simplify, we index the $N_{H}(t)$ members of $H(t)$ in ascending order based on their distance from the vertiport, such that $H(t) = \{h_1, \ldots, h_{N_{H}(t)}\}$ and $h_k$ follows $h_{k+1}$. This setup helps us recognize merging pairs as ($h_k$,\ $ h_{k+1} $), allowing eVTOLs to execute merging maneuvers into the corresponding dynamic merging point $P^{k} (t)=[x'_{h_k}(t)-L_s, (L_c-L_d) \sec\phi_i]$ between these merging pairs. Here, $ L_s $ is the minimum safe distance between the adjacent corridor eVTOLs.
However, it excludes the possibility of eVTOL $ i $ merging ahead of the first corridor eVTOL $h_1$ or behind the last $h_{N_{H}(t)}$. 
To address this, we dynamically extend the corridor eVTOL set at each time when a merging eVTOL requests a take-off. Specifically, if $ x'_{h_1}(t_i^q) \leq -2L_s $, we augment the set with a virtual eVTOL agent $I^{+}$, whose kinematic state evolves according to:
\[
V_{I^{+}} = [ v_{c}^{\min},0], \quad
x'_{I^{+}}(t) = \int_{t_i^q}^{t} V_{I^{+}}(\tau) \mathrm{d}\tau, \quad 
z'_{I^{+}} = (L_c-L_d) \sec\phi_i
\]
where $ v_{c}^{\min}$ denotes the minimum permissible horizontal cruise speed for the corridor eVTOL. Let us define $last=N_{H}(t_i^q)$ as the index of the last corridor eVTOL. Similarly, if $ x'_{h_{last}}(t_i^q) \ge -L_{o}+2L_s$, $ I^{-} $ is a ``virtual'' eVTOL with dynamics:  
\[
V_{I^{-}}(t) = V_{h_{last}}(t), \quad 
x'_{I^{-}}(t) = -L_o+ \int_{t_i^q}^{t} V_{I^{-}}(\tau) \mathrm{d}\tau, \quad 
z'_{I^{-}} = (L_c-L_d) \sec\phi_i
\]
This leads to the extended set:
% \begin{equation}\label{eq2}
% H_{e}(t) = H(t)\cup \left \{ i^{+},i^{-}\right \} 
% \end{equation}
\begin{equation}\label{eq2}
H_e(t_i^q) =
\begin{cases}
H(t_i^q) \cup \left\{ I^{+}, I^{-} \right\}, & \text{if } x'_{h_1}(t_i^q) \leq -2L_s \text{ and } x'_{h_{last}}(t_i^q) \ge -L_o + 2L_s \\
H(t_i^q) \cup \left\{ I^{+} \right\}, & \text{if } x'_{h_1}(t_i^q) \leq -2L_s \text{ and } x'_{h_{last}}(t_i^q) < -L_o + 2L_s\\ 
H(t_i^q) \cup \left\{ I^{-} \right\}, & \text{if } x'_{h_1}(t_i^q) > -2L_s \text{ and } x'_{h_{last}}(t_i^q) \ge -L_o + 2L_s \\
H(t_i^q), & \text{otherwise}
\end{cases}
\end{equation}
Similarly, the set $H_e(t_i^q)$ is sorted in ascending order based on their distance from the vertiport.
Then, we can construct all candidate merging pairs ($h_k$,\ $ h_{k+1} $) in observation zone at time $t_i^q$, with $k\in\{1,2,...,N_{H_e}(t_i^q)-1\}$ and $N_{H_e}(t_i^q)$ denotes the number of elements in set $H_e(t_i^q)$.
We then develop the HCTMM strategy to coordinate the take-off and merging processes of eVTOLs, ensuring both safety and efficiency (as depicted in \hyperref[fig:08]{Figure \ref*{fig:08}}).

\begin{figure}[h]
    \centering
    \includegraphics[width=0.8\textwidth]{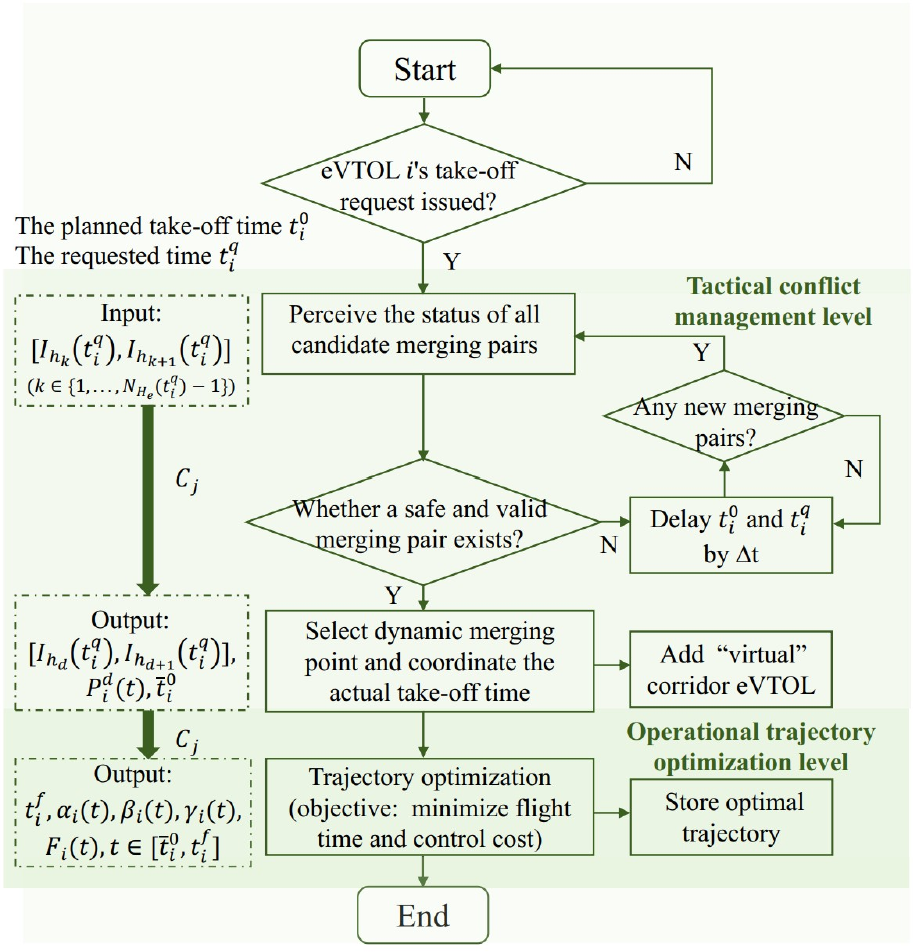}
    \caption{Framework of the HCTMM strategy.}
    \label{fig:08}
\end{figure}

In this paper, we impose the following assumptions:
\begin{assumption}
\label{ass:01}
For each corridor eVTOL $ h_k $, where $ k \in \{1,2,\dots,N_H(t)\} $, its speed is adjusted once upon entering the observation zone to ensure a safe cruising speed $ V_{h_k} $, which it then maintains constantly throughout the observation zone and TM section.
\end{assumption}
To enhance safety and reduce energy consumption, maintaining a constant speed for eVTOLs within a corridor represents an ideal state \citep{wang2021optimal,wu2022convex}. During the take-off and merging phases, constant-speed flight facilitates coordination and simplifies the problem with better predictability. 
In practice, to enforce \autoref{ass:01}, the coordinator applies a simple rule-based mechanism to regulate each corridor eVTOL's speed once, precisely when it enters the observation zone. The adjusted speed ensures safe separation from the preceding vehicle within both the observation and TM sections:
\begin{equation}
    v_{h_k}^{x'} = \min\left( v_{h_k}^{x'}(t_e), \frac{v_{h_{k-1}}^{x'}(L_m+L_o - L_s)}{L_m - x'_{h_{k-1}}(t_e)} \right), \quad v_{h_k}^{z'} = v_{h_k}^{z'}(t_e)
\end{equation}
Here, $ t_e $ denotes the time when corridor eVTOL $ h_k $ reaches the entrance of the observation zone and $L_m$ denotes the width of the TM section.

\begin{assumption}\label{asm:no delay}
    There are no communication delays or errors between eVTOLs. All eVTOLs comply with the control policy issued by the coordinator. % is properly functioning. Thus aircraft can execute instructions properly. 
\end{assumption}
\autoref{asm:no delay} ensures reliable communication between eVTOLs and the coordinator. This is a typical setting in the literature on eVTOLs and CAV trajectory planning \citep{zhou2018optimal,wu2022convex,jiang2022coordination}.

\section{Tactical conflict management}
\label{sec: Tactical conflict management}

At the tactical level, the vertiport coordinator selects dynamic merging points $P_{i}^{d}(t)$ and adjusts the actual take-off times $\bar{t}_{i}^{0}$ by leveraging the status of candidate merging pairs and the trajectories $C_j$ of eVTOLs $j \in \mathcal{J}$. Here, $\mathcal{J}$ represents the set of eVTOLs operating within the same TM section as eVTOL $i$ that have not yet completed the merging process at the planned take-off time $t_i^0$.
In this section, two key principles for scheduling are investigated.
The principle of safety is essential for merging control and is introduced in the following subsection.
Building on the principle of safety, the principle of validity—critical for coordinating the merging point and take-off time—is thoroughly examined.
% The principles of safety and validity are essential for scheduling the merging point and take-off time, which are investigated in this section.
Then we present a comprehensive algorithm for dynamic merging point selection and take-off time coordination for the merging eVTOL $i$. 

% \subsection{Merging point selection and take-off time coordination}
% \subsection{Principle of safety}
% \label{sec:4.1}
\subsection{Principles of safety and validity}
\label{sec:4.1}

To meet the safety separation criteria for corridor flight operations, the length of the merging pair $L_k(t)$ should be maintained above $2L_s$ to ensure both the safe merging maneuver of eVTOL $i$ and its conflict-free exit from the TM section. 
Specifically, the safety principle is given by $ L_k(t)\ge 2L_s, \; t \in (t_i^f, T_k^{ter}) $, where $ t_i^f $ represents the completion of the take-off and merging processes and $ T_k^{ter} $ denotes the time it takes for the corridor eVTOL $h_k$ to reach the endpoint ($ter$) of TM section, with $p^{ter}=[L_{m},(L_c-L_d) \sec\phi_i]$. Given \hyperref[ass:01]{Assumption \ref*{ass:01}}, we evaluate $ T_k^{ter} $ and $ L_k(t) $ as follows:

 \begin{equation}
    T_k^{ter} = \frac{L_m- x'_{h_k}(t)}{v_{h_k}^{x'}} + t, \; x'_{h_k}(t) \in [-L_o, L_m]
    \label{eq:10}
 \end{equation}
 \begin{equation}
\label{eq:A11}
    L_k(t) = x'_{h_k}(t) - x'_{h_{k+1}}(t) 
\end{equation}

There are three scenarios regarding the variation in the length of the merging pair:
\begin{enumerate}
    \item $ v_{h_k}^{x'} > v_{k_{h+1}}^{x'} $: $ L_{k}(t) $ is monotonically increasing. If $L_{k}(T_{k}^{ter}) \ge 2L_s$, the merging pair ($h_k$,\ $h_{k+1}$) satisfies the safety principle when the corridor eVTOL $h_k$ reaches the endpoint ($ter$), with $t_i^f=T_{k}^{ter}$.
    \item $ v_{h_k}^{x'} < v_{h_{k+1}}^{x'} $: $ L_{k}(t) $ is monotonically decreasing. If $ L_k(T_k^{ter}) \ge 2L_s $, the merging pair ($h_k$,\ $h_{k+1}$) meets the safety principle within the entire TM section.
    % {\color{red} The endpoint $ter$ is the only point of judgment.} 
    \item $ v_{h_k}^{x'} = v_{h_{k+1}}^{x'} $: $ L_{k}(t) $ remains constant.
\end{enumerate}
As discussed, if $L_{k}(T_{k}^{ter}) \ge 2L_s$, then the merging pair ($h_k$,\ $h_{k+1}$) is considered to be safe for supporting merging and conflict-free exit.

Due to the limits of maximum speed, thrust, and safe distance, it is necessary to estimate whether the merging eVTOL $i$ can accomplish the take-off and merging processes within the TM section. The assessment of these constraints is termed the principle of validity, which will be specified later in this section.  %\autoref{sec: Principles of validity}.
Based on the principles of safety and validity, three types of merging pairs are defined as follows \citep{xue2022platoon}: 
\begin{enumerate}
\item A \textit{safe and valid merging pair} refers to merging pair ($h_k$,\ $h_{k+1}$) that satisfies the following conditions: (1) the length of merging pair ($h_k$,\ $h_{k+1}$) is sufficient to support safe merging and conflict-free exit operations, and (2) the merging pair ($h_k$,\ $h_{k+1}$) allows eVTOL $i$ to complete both the take-off and merging processes before the pair leaves the TM section.
\item A \textit{failed merging pair} occurs when the merging pair satisfies condition (1) but not condition (2).
\item An \textit{unsafe merging pair} denotes a merging pair with insufficient length to satisfy condition (1).
\end{enumerate}
% \vspace{-1.38em}

% \subsection{Principle of validity}
% \label{sec: Principles of validity}
From the view of validity assessment for the merging pair ($h_k$,\ $h_{k+1}$), the merging eVTOL $i \in \mathcal{I}$ should finish the take-off and merging processes before the corridor eVTOL $h_k$ exits the
TM section. To achieve this, we first present the following definition:

\begin{definition}
\label{def:1}
    A candidate merging pair ($h_k$,\ $ h_{k+1} $) is said to be valid if the following conditions are satisfied:
    (i) There exists at least one feasible point $p_f = [x'_f,z'_f] $ satisfying $z'_f= (L_c-L_d) \sec \phi_i $ and $ 0 \leq x'_f \leq L_m $;
    (ii) The travel time $T_k^{f}$ for corridor eVTOL $h_k$ to reach the point $p_f$ must satisfies $ T_k^{f} \ge \bar{t}_{ki}^{lf} $. Here, $ \bar{t}_{ki}^{lf} $ represents the minimum time for merging eVTOL $i$ to reach the position $ p_{lf} = [x'_f-L_s, (L_c-L_d) \sec \phi_i] $.
\end{definition}

\begin{remark}
    It is notable that the minimum time $ \bar{t}_{ki}^{lf}$ is subject to the conflict avoidance requirements imposed by corridor eVTOLs in merging pair $ (h_{k} $,\ $ h_{k+1}) $, as well as eVTOLs $j\in \mathcal{J}$ that coexist in the same TM section and have not completed the merging process. These requirements ensure that the merging eVTOL $i$ does not lead to conflicts in the take-off and merging processes.
\end{remark}

%\hyperref[def:1]{Definition \ref*{def:1}} establishes the sufficiency conditions for determining the validity of merging pair ($h_k$,\ $ h_{k+1} $). 
To rigorously assess the validity of the candidate merging pair per \hyperref[def:1]{Definition \ref*{def:1}}, we derive the necessary and sufficient condition in \hyperref[pro:01]{Proposition \ref{pro:01}}.
\begin{proposition} 
\label{pro:01} 

Consider corridor eVTOL $h_k$, whose motion satisfies \hyperref[ass:01]{Assumption~\ref{ass:01}} and adheres to the merging point design criteria specified in \autoref{sec:Trajectory optimization for take-off and merging control}, if a merging pair $(h_k,\ h_{k+1})$ satisfies the safety principle defined in \autoref{sec:4.1}, then the condition $T_k^{ter} \ge \bar{t}_{ki}^{a}$ is necessary and sufficient for $(h_k,\ h_{k+1})$ to be a valid merging pair.
Here, $ \bar{t}_{ki}^{a} $ denotes the minimum time required for eVTOL $ i $ to complete take-off and merging at the farthest allowable merging point $ p^{a} = [L_m - L_s, (L_c-L_d) \sec \phi_i] $.
This time is subject to the same constraints as $ \bar{t}_{ki}^{lf}$.
\end{proposition}

\begin{proof}
We now prove that the condition $T_k^{ter} \geq \bar{t}_{ki}^{a}$ in \hyperref[pro:01]{Proposition~\ref*{pro:01}} constitutes a necessary and sufficient condition to support the candidate merging pair ($h_k$,\ $h_{k+1}$) is valid per \hyperref[def:1]{Definition \ref*{def:1}}.

\textbf{Sufficiency  ($\Rightarrow$):} 
 We can identify the endpoint $p^{ter}$ of the TM section as the point $p_f$ and $ p^{a}$ as $ p_{lf}$. If $T_k^{ter} \geq \bar{t}_{ki}^{a}$ holds. Then condition (ii) of \hyperref[def:1]{Definition \ref*{def:1}} is satisfied. Since the coordinates of the endpoint are $p^{ter}=[L_{m}, (L_c-L_d) \sec\phi_i]$, condition (i) is also satisfied, thereby the merging pair ($h_k$,\ $ h_{k+1} $) satisfies \hyperref[def:1]{Definition \ref*{def:1}}.

\textbf{Necessity ($\Leftarrow$):} 
Conversely, for a valid merging pair ($h_k$,\ $ h_{k+1} $), the merging eVTOL $i$ will be able to merge into the corresponding merging point $P^k(t_i^f)=[{x'}^k(t_i^f),(L_c-L_d) \sec\phi_i]$ at time $t_i^f$. 
According to the design principle of merging points in \autoref{sec:Trajectory optimization for take-off and merging control}, the corridor eVTOL $h_k$ reaches the position $[{x'}^k(t_i^f)+L_s,(L_c-L_d) \sec\phi_i]$ at time $t_i^f$.
Given $ 0 \leq {x'}^k(t_i^f)+L_s \leq L_m$ and \hyperref[ass:01]{Assumption \ref*{ass:01}}, we have $ t_i^f + (L_m - {x'}^k(t_i^f)-L_s)(v_{h_k}^{x'})^{-1} = T_k^{ter} $. 
From the definition of $\bar{t}_{ki}^{a}$ and eVTOL $ i $ maintains the constant speed $V_{h_k}$ after completing the merging process, the shortest time for eVTOL $ i $ to reach the farthest allowable merging point satisfies 
$\bar{t}_{ki}^{a}\leq t_i^f+(L_m- L_s-{x'}^k(t_i^f) )(v_{h_k}^{x'})^{-1}= T_k^{ter}$. 
% Thus, the merging pair ($h_k$, $h_{k+1}$) guarantees the satisfaction of the condition $T_k^{\mathrm{ter}} \geq \bar{t}_{ki}^{a}$ in \hyperref[pro:01]{Proposition~\ref*{pro:01}}.

Hence, $ T_k^{{ter}} \geq \bar{t}_{ki}^{a} $ is both necessary and sufficient for the validity of the candidate merging pair $ (h_k,\ h_{k+1}) $. The proof is completed.
\end{proof} 

\subsection{Optimal merging point selection and take-off time coordination}
\label{sec: 3.2}

To support the validity assessment of a merging pair, the minimum time $ \bar{t}_{ki}^{a} $ needs to be determined. This is achieved by dividing the entire operation into two sequential phases: vertical take-off and take-off climb, as described in \autoref{sec: Take-off airspace design}. The shortest time required for the vertical take-off phase, denoted by $t_i^v$, is readily obtainable as follows:
\begin{equation}
    t_i^v=t_i^0+\frac{m_i  v_{TOSS}}{F_i^{\max}-m_ig}+\frac{L_d}{ v_{TOSS}}-\frac{m_i v_{TOSS}}{2(F_i^{\max}-m_ig)}
    %\frac{D_i^z-\frac{v_{T|max}}{\frac{2F_i^{\max}-m_ig}{m_i}}{v_{T|max}}
\end{equation}
where the positive parameter $F_i^{\max}$ represents the maximum available net thrust. While the shortest time for the take-off climb phase, starting from the transition point $D_i$ and merging into the farthest allowable merging point $p^a$ is difficult to obtain directly. Therefore, we formulate an optimal control problem to generate the shortest-time trajectory for the take-off climb and merge processes. This formulation directly determines the minimum achievable time $\bar{t}_{ki}^{a}$.

1) \textit{Simplified eVTOL dynamics and mechanical constraints:} We denote the longitudinal position, vertical position, longitudinal speed, vertical speed, tip-path-plane pitch angle, mass and net thrust of eVTOL $i$ at time $t$ as $x'_i(t)$, $ z'_i(t)$, $ v_{i}^{x'}(t)$, $v_{i}^{z'}(t)$, $\theta_i(t)$, $m_i$ and $F_i(t)$, respectively. 
Through the structured take-off airspace design, we can simplify the system dynamics as follows (refer to \autoref{sec: Appendix}): 
\begin{subequations}
\begin{gather}
\dot{x'_i}(t)=v_{i}^{x'} (t)\\
\dot{z'_i}(t)=v_{i}^{z'} (t)\\
\dot{v}_{i}^{x'}(t)=\frac{F_i(t)}{m_i}\sqrt{1 - \left(\frac{m_i g \sin\phi_i}{F_i}\right)^2}\sin(\theta_i(t))\\
\dot{v}_{i}^{z'}(t)=\frac{F_i(t)}{m_i}\sqrt{1 - \left(\frac{m_i g \sin\phi_i}{F_i}\right)^2}\cos(\theta_i(t))-g \cos\phi_i
\end{gather}
\label{eq:3}
\end{subequations}

We denote the maximum total speed of eVTOL $i$ as $V_i^{\max}$ $(V_i^{\max} > 0)$. Mechanical constraints of each eVTOL can be cast as
\begin{subequations}
\begin{gather}
\sqrt{(v_{i}^{x'}(t))^2+ (v_{i}^{z'}(t))^2}  \le V_i^{\max},  \hspace{2em} \forall t\in[t_i^v,\bar{t}_{ki}^{a}]  \label{eq:sub_6a} \\
-\pi \le {\theta_i}(t)   \le \pi,\hspace{2em} \forall t\in[t_i^v,\bar{t}_{ki}^{a}]\label{eq:sub_6b}\\
0\le {F_i}(t) \le \sqrt{(F_i^{\max})^2-(m_i g \sin\phi_i)^2},\hspace{2em} \forall t\in[t_i^v,\bar{t}_{ki}^{a}]\label{eq:sub_6c}
\end{gather}
\label{eq:4}
\end{subequations}
Besides, to ensure realistic vertical climb performance during the take-off climb and merging processes, an upper bound for the rate of climb (RoC) constraint is introduced. The RoC represents the vertical component of the eVTOL's speed \citep{anderson1999aircraft} and is defined as $v_{i}^{z'}(t)\cos\phi_i$ in the transformed $X_i'D_iZ_i'$ coordinate system. The constraint can be written as
\begin{gather}
    v_{i}^{z'}(t)\cos\phi_i<RoC_{\max},\hspace{2em} \forall t\in[t_i^v,\bar{t}_{ki}^{a}]\label{eq:4.5}
\end{gather}
where $RoC_{\max}$ is the maximum rate of climb.

2) \textit{Safety constraints:} Aircraft safety is the top concern for take-off climb and merging processes. Within TM sections, the primary safety concern for eVTOL $i$ is the potential collision with corridor eVTOLs $h_k$, $h_{k+1}$ in the merging pair, as well as with eVTOLs $j \in \mathcal{J}$ that have not yet completed the merging process and remain within the same TM section.
In light of \cite{wu2022convex}, we introduce $s_f$ to indicate the predefined safe distance between two eVTOLs. For the minimum safe gap between the adjacent corridor eVTOLs $L_s$, we have $ L_s \geq s_f $. Now, we define collision avoidance constraints for eVTOL $i$ during the take-off climb and merging processes as
\begin{equation}\label{eq:6}
(x'_i(t)-x'_{n}(t))^2+(z'_i(t)-z'_{n}(t))^2\ge s_f^2,\hspace{1.5em} \forall t\in[t_i^v,\bar{t}_{ki}^{a}],\; n\in \mathcal{J} \cup\{h_k,h_{k+1}\}
\end{equation}
Building on the TM section design concept outlined in \autoref{sec: Take-off airspace design}, we restrict operation within the designated sections to mitigate obstacle avoidance challenges across distinct corridor-vertiport pairs. The trajectory range constraints can be expressed as follows: 
\begin{subequations}
\begin{gather}
0\le {x'_i}(t) \le L_m,\hspace{1.5em} \forall t\in[t_i^v,\bar{t}_{ki}^{a}]\\
0\le {z'_i}(t) \le (L_c-L_d) \sec\phi_i,\hspace{1.5em} \forall t\in[t_i^v,\bar{t}_{ki}^{a}]\label{eq:sub_c8}\\
{x'_i}(t)\tan \Theta_{dep} \le {z'_i}(t)\cos\phi_i,\hspace{1.5em} \forall t\in[t_i^v,\bar{t}_{ki}^{a}]
\end{gather}
\label{eq:5}
\end{subequations}
where $\Theta_{dep}$ is obstacle limitation surface \citep{EasyAccess2024}.

3) \textit{Boundary constraints:} 
The starting and terminal positions of the shortest-time trajectory are the transition point $D_i$ and the farthest allowable merging point $p^a$, respectively. 
The shortest-time trajectory continues until the eVTOL $i$ merges into the candidate merging pair $(h_k,\ h_{k+1})$ at point $p^a$, while maintaining the same speed as the corridor eVTOL $h_{k}$. These constraints can be expressed as follows:
\begin{equation}
x'_i(t_i^v)=0, \;\;\; z'_i(t_i^v)=0,\;\;\; v_{i}^{x'}(t_i^v)=0,\;\;\; v_{i}^{z'}(t_i^v)=v_{TOSS}
\label{eq:7}
\end{equation}
\begin{equation}
x'_i(\bar{t}_{ki}^{a})=L_m-L_s,\;\;\; z'_i(\bar{t}_{ki}^{a})=(L_c-L_d) \sec\phi_i,\;\;\;v_{i}^{x'}(\bar{t}_{ki}^{a})=v_{h_k}^x,\;\;\; v_{i}^{z'}(\bar{t}_{ki}^{a})=0
\label{eq:B1}
\end{equation}

Summarizing \eqref{eq:3}-\eqref{eq:B1}, we define a free-terminal-time optimal control problem for the shortest-time candidate trajectory as follows:
\begin{problem}\label{prob:1}
\begin{subequations}
\begin{gather}
    \min_{\substack{F_{i}(t), \theta_i(t)}} J_i = \bar{t}_{ki}^{a}-t_i^v\\
\text{subject to: \eqref{eq:3}-\eqref{eq:B1}}
\end{gather}
\end{subequations}
\end{problem}
 
The decision variables are:
the net thrust of eVTOL $i$, i.e., $F_i(t)$, and the tip-path-plane pitch angle, i.e., $\theta_{i}(t)$. The constraints \eqref{eq:3}-\eqref{eq:B1} include system dynamics, aircraft mechanical limitations, and safety constraints, respectively. 
As discussed, the solution $ \bar{t}_{ki}^{a}$ forms the basis for the validity assessment of the candidate merging pair $(h_k,\ h_{k+1})$. %, which is carried out in step 9 of \hyperref[alg:01]{Algorithm \ref*{alg:01}}.

\begin{algorithm}[!h]
\caption{Merging point selection and take-off time coordination for merging eVTOL $i \in \mathcal{I}$}
\label{alg:01}
        \KwIn{  $t_{i}^{0}$: The planned take-off time\\
                \setlength{\parindent}{3.3em}$t_{i}^{q}$: The requested take-off time\\
                \setlength{\parindent}{3.1em}$\Delta  t$: The predefined take-off delay interval\\
                    \setlength{\parindent}{3.1em}$\mathcal{J}$: The set of eVTOLs in the same TM section as eVTOL $i$ that have not merged at time $t_i^0$ }
	\KwOut{
                $[I_{h_d}(t_i^q),\ I_{h_{d+1}}(t_i^q)]$: The status of per-target merging pair at time $t_i^q$, where:\\
                \setlength{\parindent}{6.2em}$I_{h_d}(t)=\{[x'_{h_d}(t),z'_{h_d}(t)],[v_{h_d}^{x'}(t),v_{h_d}^{z'}(t)]\} $\\
                %$(h_d,h_{d+1})$: the per-target merging pair \\
                \setlength{\parindent}{4.2em}$P_i^d(t)$: The dynamic merging point, $P_i^d(t)=[{x'}_i^d(t),(L_c-L_d) \sec\phi_i]$ \\
                \setlength{\parindent}{4.2em}$\bar{t}_{i}^{0}$: The actual take-off time}    

\Begin{
    Initialize: Determine $H(t_{i}^{q}), N_{H}(t_{i}^{q}), H_{e}(t_{i}^{q}), N_{H_e}(t_{i}^{q}), C_j$
    
    Perceive the status of all candidate merging pairs at time $t_{i}^{q}$: $[I_{h_k}(t_i^q),\ I_{h_{k+1}}(t_i^q)]$, where $k \in \{1,2,...,N_{H_e}(t_{i}^{q})-1\}$
    
    \For{$k=1$ \KwTo $N_{H_e}(t_{i}^{q})-1$}{
        Predict the time $T_{k}^{ter}$ for the corridor eVTOL $h_k$ to reach the endpoint by \eqref{eq:10}
        
        Calculate the length of merging pair $L_{k}(T_{k}^{ter})$ by \eqref{eq:A11}
        
        \If{$L_{k}(T_{k}^{ter}) \ge 2L_s$}{
                Merging pair ($h_k$,\ $ h_{k+1} $) satisfies the safety principle
                
                Calculate the time $\bar{t}_{ki}^{a}$ by \autoref{prob:1}
                
                \If{$\bar{t}_{ki}^{a} \le T_{k}^{ter}$}{
                    Merging pair ($h_k$,\ $ h_{k+1} $) satisfies the validity principle % principle of validity per \autoref{def:1}
                    
                    Determine the per-target merging pair $(h_d,\ h_{d+1})=(h_k,\ h_{k+1})$ and the dynamic merging point $P_i^d(t)=P^k(t)$, set the actual take-off time: $\bar{t}_{i}^{0} = t_{i}^{0}$

                    Add a ``virtual''  corridor eVTOL at the dynamic merging point
                    
                    \Return result, break
                }
            }
    }
    Delay the planned take-off time until $t_{i}^{0}+\Delta  t$, determine $t_{i}^{q}+\Delta  t$, $H_e(t_{i}^{q}+\Delta  t)$ and $ N_{H_e}(t_{i}^{q}+\Delta  t)$

    Define $N_{new}= N_{H_e}(t_{i}^{q}+\Delta  t)$ and $N_{old}= N_{H_e}(t_{i}^{q})$
    
    Perceive the status of eVTOL $h_{N_{new}}$ at time $t_{i}^{q}+\Delta  t$: $I_{h_{N_{new}}}(t_{i}^{q}+\Delta  t)$
    
     $t_{i}^{q} = t_{i}^{q} + \Delta  t$
     
     $t_{i}^{0} = t_{i}^{0} + \Delta  t$   
    
    \If{$x'_{h_{N_{new}}}(t_{i}^{q})<x'_{h_{N_{old}}}(t_{i}^{q})$}{
        \Return to Step 2
    }
    Return to Step 15
}
\end{algorithm}

\hyperref[alg:01]{Algorithm \ref*{alg:01}} outlines the process for selecting the dynamic merging point and coordinate take-off time per the principles discussed above. Considering the planning efficiency of vertiports, each merging eVTOL must submit a request at least $ T $ seconds before its planned take-off time $ t_i^0 $. Here, $ T $ denotes the planning horizon, and $ t_i^q $ represents the requested take-off time for merging eVTOL $ i $.
Based on the status information $[I_{h_k}(t),\ I_{h_{k+1}}(t)]$, candidate merging pairs $ (h_k,\ h_{k+1})$ are sequentially evaluated for safety and validity, starting from the first pair (i.e., the merging pair $ (h_1,\ h_2)$).
% the safety and validity of each corresponding gap $g_k$ are sequentially evaluated, starting from the first gap $g_1$.
Steps 5-7 involve the safety assessment. 
%This safety principle is checked from the merging completion time to the moment the vehicle exits the section. 
% If $ L_{k}(T_{k}^{ter}) $ is greater, the gap $g_k$ meets the safe principle.
Steps 9-10 involve the valid assessment.
To maximize the operational efficiency for eVTOL $i$, the first safe and valid merging pair in the observation zone is set as the per-target merging pair $(h_d,\ h_{d+1})=(h_k,\ h_{k+1})$.  This rule prioritizes early merging opportunities, which contribute to less steep climb profiles and reduced take-off delays. As will be validated in our experiments, this heuristic selection yields results consistent with those obtained by rule-based exhaustive search aiming to minimize total flight time and control cost. Then, we choose the corresponding dynamic merging point as the target merging point $ P_{i}^{d}(t) = P^{k}(t) $ for the merging eVTOL $ i $.
Meanwhile, the merging eVTOL $ i $ is able to take-off at the planned time, implying that $ \bar{t}_{i}^{0} = t_{i}^{0} $. If none of the candidate merging pairs satisfy safety and viability principles, it may be necessary to delay the planned take-off time of eVTOL $i$, allowing new eVTOLs to enter the observation zone and creating new pairs for selection. Steps 15-20 execute the aforementioned operations.

\section{Operational trajectory optimization}
\label{sec: Operational deconfliction control}

The actual take-off time $\bar{t}_{i}^{0}$, the status of per-target merging pair $[I_{h_d}(t_i^q),\ I_{h_{d+1}}(t_i^q)]$, and the dynamic merging point $P_i^d(t)=[{x'}^d_i(t),(L_c-L_d)\sec\phi_i]$, as derived in \autoref{sec: Tactical conflict management}, are fed to the operational trajectory optimization. 
This section presents the optimization strategy for providing a control cost and time-optimal take-off and merging trajectory at the operational level.

For the vertical take-off phase, the eVTOL is required to perform a rapid vertical ascent to minimize exposure time in low-altitude high-risk zones and reach the take-off safety speed $V_{TOSS}$, ensuring stable flight conditions \citep{EasyAccess2024}.
Accordingly, the trajectory in this phase is characterized by applying the maximum allowable acceleration until $V_{TOSS}$ is reached, followed by the constant-velocity ascent to the predefined transition point altitude. Therefore, the time to reach the transition point, denoted as $\bar{t}_{i}^{v}$, is computed similarly as described in \autoref{sec:4.1}.
For the take-off climb phase, in light of \cite{wu2022convex,armijos2024cooperative}, we consider two objectives for the trajectory optimization problem for each eVTOL $i$. First, we minimize the completion time $t_i^f$. Second, we minimize the control cost for the eVTOL to accomplish the take-off climb and merging processes. In this paper, we define the control cost in terms of net thrust, which is one of the most commonly used cost functions in eVTOL trajectory optimization \citep{wu2022convex, wu2024convex, shen2023convex, zhang2024collision}.
Similar to the shortest-time trajectory investigated in \autoref{sec:4.1}, the take-off climb and merging trajectory must satisfy the vehicle and safety constraints \eqref{eq:3}-\eqref{eq:5} within the time domain $[\bar{t}_{i}^{v},t_{i}^{f}]$.
The completion of take-off climb and merging processes is marked by eVTOL $i$ reaching the merging point $P_{i}^{d} (t_i^f)$ and keeping the same speed of the corridor eVTOL $h_d$ at time $t_i^f$.
Accordingly, boundary constraints for the process can be expressed as follows:
\begin{equation}
x'_i(t_i^f)={x'}_i^d(t_i^f),\;\;\; z'_i(t_i^f)=(L_c-L_d)\sec\phi_i,\;\;\;
v_{i}^{x'}(t_i^f)=v_{h_d}^{x'},\;\;\; v_{i}^{z'}(t_i^f)=0
\label{eq:12}
\end{equation}
The horizontal position of merging point ${x'}_i^d(t_i^f)$ can be predicted based on $x'_{h_d}(t_i^q)$, which is included in status $I_{h_d}(t_i^q)$. ${x'}_i^d(t_i^f)$ is expressed as follows.
\begin{equation}
         {x'}^{d}_i(t_i^f)= x'_{h_d}(t_{i}^{q})+v_{h_d}^{x'}(t_i^f-t_{i}^{q})-L_s
         \label{eq:13}
    \end{equation}
Based on the TM section design concept outlined in \autoref{sec: Take-off airspace design}, the trajectories of merging eVTOLs operating across different vertiport-corridor pairs will inherently maintain safe separation. Consequently, the collaborative trajectory planning problem of multiple merging eVTOLs can be transformed into a series of individual trajectory optimization problems, improving planning efficiency while ensuring safety in trajectory execution.
The control cost and time-optimal take-off climb and merging trajectory optimization problem for each merging eVTOL $i \in \mathcal{I}$ can be formulated as an optimal control problem with free terminal time and terminal states that depend on the terminal time, as described in \autoref{prob:2}.
The feasibility of \autoref{prob:2} is guaranteed by evaluating the principles of safety (as specified in \eqref{eq:10} and Step 7 of \hyperref[alg:01]{Algorithm~\ref*{alg:01}}) and validity (as specified in \autoref{prob:1} and Step 10 of \hyperref[alg:01]{Algorithm~\ref*{alg:01}}).

\begin{problem}\label{prob:2}
\begin{subequations}
\begin{gather}
    \min_{\substack{F_{i}(t), \theta_i(t)}} J_i =\int_{\bar{t}_{i}^{v}}^{t_{i}^{f}} \bigg[\frac{1}{2}\bigg(\frac{F_{i}(t)}{m_i}\bigg)^{2}+\lambda\bigg] dt \\
    \text{subject to: \eqref{eq:3}, \eqref{eq:7}} ,\eqref{eq:12}, \eqref{eq:13}\\
    0\le {x'_i}(t) \le L_m,\hspace{1.5em} \forall t\in[\bar{t}_{i}^{v},t_{i}^{f}] \label{eq:sub_c1} \\
    0\le {z'_i}(t) \le (L_c-L_d) \sec\phi_i,\hspace{1.5em} \forall t\in[\bar{t}_{i}^{v},t_{i}^{f}] \label{eq:sub_c2} \\
    {x'_i}(t)\tan \Theta_{dep} \le {z'_i}(t)\cos\phi_i ,\hspace{1.5em} \forall
    t\in[\bar{t}_{i}^{v},t_{i}^{f}]  \label{eq:sub_c7} \\
    \sqrt{(v_{i}^{x'}(t))^2+ (v_{i}^{z'}(t))^2}  \le V_i^{\max},\hspace{2em} \forall t\in[\bar{t}_{i}^{v},t_{i}^{f}]  \label{eq:sub_c3} \\
    v_{i}^{z'}(t)\cos\phi_i<RoC_{\max},\hspace{2em} \forall t\in[\bar{t}_{i}^{v},t_{i}^{f}]  \label{eq:sub_c6}\\
    -\pi \le {\theta_i}(t)   \le \pi,\hspace{2em} \forall t\in[\bar{t}_{i}^{v},t_{i}^{f}] \label{eq:sub_c4} %\\ 
\end{gather}
\end{subequations}
\newpage
\begin{subequations}
\begin{gather}
    0\le {F_i}(t) \le \sqrt{(F_i^{\max})^2-(m_i g \sin\phi_i)^2},\hspace{2em} \forall t\in[\bar{t}_{i}^{v},t_{i}^{f}]  \label{eq:sub_c5} \tag{19i} \\
    (x'_i(t)-x'_{n}(t))^2+(z'_i(t)-z'_{n}(t))^2\ge s_f^2,\hspace{1.5em} \forall  t\in[\bar{t}_{i}^{v},t_{i}^{f}],\; n\in \mathcal{J} \cup\{d,d+1\}  \label{eq:sub_c7} \tag{19j}
\end{gather}
\end{subequations}
\end{problem}
\noindent where $\lambda$ is the time-to-control value conversion coefficient. The variables and constraints are the same as those of \autoref{prob:1}.

% Note that the system dynamics in \autoref{prob:1} and \autoref{prob:2} are nonlinear. The nonlinearity leads to significant complexity and computational time.
% To address this, we first introduce the following new variables:
The formulation of \autoref{prob:1} and \autoref{prob:2} can be made more concise by introducing appropriately defined variables.
For this purpose, we introduce the following new variables:
\begin{equation}
\label{eq:456}
    u_{i}^{x'}(t)=\frac{F_i(t)}{m_i}\sqrt{1 - \left(\frac{m_i g \sin\phi_i}{F_i}\right)^2}\sin(\theta_i(t)),\quad u_{i}^{z'}(t)=\frac{F_i(t)}{m_i}\sqrt{1 - \left(\frac{m_i g \sin\phi_i}{F_i}\right)^2}\cos(\theta_i(t))
\end{equation}
Along with the above new variables, the control constraints \eqref{eq:sub_c4} and \eqref{eq:sub_c5} (or \eqref{eq:sub_6b} and \eqref{eq:sub_6c}) can be represented by:
\begin{equation}
\label{eq:123}
(m_iu_{i}^{x'}(t))^2+(m_iu_{i}^{z'}(t))^2  \le(F_i^{\max})^2- (m_i g \sin\phi_i)^2
\end{equation}
Subsequently, the original system dynamics described in \eqref{eq:3} are transformed into the following form:
\begin{subequations}
\begin{gather}
\dot{x'_i}(t)=v_{i}^{x'}(t)\\
\dot{z'_i}(t)=v_{i}^{z'}(t)\\
\dot{v}_{i}^{x'}(t)=u_{i}^{x'}(t)\\
\dot{v}_{i}^{z'}(t)=u_{i}^{z'}(t)-g \cos\phi_i
\end{gather}
\label{eq:333}
\end{subequations}
\autoref{prob:1} and \autoref{prob:2} can be reformulated as \autoref{prob:6} and \autoref{prob:4}, respectively, as follows:
\begin{problem}\label{prob:6}
\begin{subequations}
\label{eq:A}
\begin{gather}
     \min_{\substack{u_{i}^{x'}(t), u_{i}^{z'}(t)}} J_i =\bar{t}_{ki}^{a} - t_i^v \\
     \text{subject to: \eqref{eq:sub_6a}, \eqref{eq:4.5}-\eqref{eq:B1}, \eqref{eq:123}}, \eqref{eq:333}
\end{gather}
\end{subequations}
\end{problem}
\begin{problem}\label{prob:4}
\begin{subequations}
\label{eq:19}
\begin{gather}
     \min_{\substack{u_{i}^{x'}(t), u_{i}^{z'}(t)}} J_i =\int_{\bar{t}_{i}^{v}}^{t_{i}^{f}} \frac{1}{2}\bigg[ (u_{i}^{x'}(t))^2+(u_{i}^{z'}(t))^2+2\lambda\bigg] dt \\
     \text{subject to: \eqref{eq:7}} ,\eqref{eq:12}, \eqref{eq:13},\eqref{eq:sub_c1}-\eqref{eq:sub_c6},\eqref{eq:sub_c7},\eqref{eq:123},\eqref{eq:333}
\end{gather}
\end{subequations}
\end{problem}

Inspired by \cite{gao2021optimal,wu2022convex}, \autoref{prob:6} and \autoref{prob:4}
% can be analytically obtained, as shown in \ref{sec: Appendix03}, %\nameref{sec: Appendix03},
% through standard Hamiltonian analysis similar to OCPs formulated and solved in \citep{armijos2024cooperative,jiang2022coordination}.
% Motivated by \cite{wu2022convex,gao2021optimal}, they 
can be solved using the pseudo-spectral method, which has the advantages of low initial sensitivity and large convergence radius. The state-of-the-art General Pseudospectral Optimal Control Software (GPOPS-\uppercase\expandafter{\romannumeral2}) can be used to solve \autoref{prob:6} and \autoref{prob:4} \citep{patterson2014gpops}.

% However, we recognize that it remains a new challenge to directly control the net thrust in $z'$-axis (unless $\phi_i=0^\circ$), which is perpendicular to the corridor direction. For eVTOLs, a more feasible control approach involves the precise regulation of acceleration in the $x$, $y$, and $z $ directions of an inertial reference frame \citep{li2024event,shen2023convex}. 

% To better guide the merging eVTOL in executing its optimal trajectory, 
After achieving the optimal solutions of \autoref{prob:4}, it remains a new challenge to control the net thrust in the $Z_i'$-axis directly. To address this challenge, we project the solutions of \autoref{prob:4} back into the inertial reference frame $X_iY_iZ_i$ and original control variable, i.e., net thrust $ F_i $ and Euler angles $ (\alpha_i, \beta_i, \gamma_i) $.
% First, the solutions of thrust are projected from two dimensions to three-dimensional acceleration components.
% \begin{equation}
% \label{eq:16}
%     u_{x_i}(t)=\frac{F_{x_i}(t)}{m},\quad u_{y_i}(t)=\frac{F_{z'_i}(t)}{m}sin\phi_i, \quad u_{z_i}(t)=\frac{F_{z'_i}(t)}{m}cos\phi_i
% \end{equation}
\begin{subequations}
\begin{gather}
\label{eq:16}
F_i=\sqrt{(m_iu_{i}^{x'})^2+(m_iu_{i}^{z'})^2+(m_i g \sin\phi_i)^2}    \\
\tan\theta_i=\frac{u_i^{x'}}{u_i^{z'}}\\
\frac{F_i}{m_i}[\cos\phi_i(\cos\alpha_i sin\beta_i \sin\gamma_i - \sin\alpha_i \cos \gamma_i)-\sin\phi_i \cos\beta_i \cos\alpha_i]=-g\sin\phi_i\\
\sin\phi_i(\cos\alpha_i \sin\beta_i \sin\gamma_i - \sin\alpha_i \cos\gamma_i)+ \cos\phi_i \cos\beta_i \cos\alpha_i=\sqrt{1 - \left(\frac{m_i g \sin\phi_i}{F_i}\right)^2}\cos\theta_i\\
\cos\alpha_i \sin\beta_i \cos\gamma_i + \sin\alpha_i \sin\gamma_i=\sqrt{1 - \left(\frac{m_i g \sin\phi_i}{F_i}\right)^2}\sin\theta_i
\end{gather}
\end{subequations}
% If $ \phi_i = 0^\circ $, \eqref{eq:16d} is replaced by \eqref{eq:e}.
% Similarly, the solutions for state variables are projected in the same manner.
% \begin{subequations}
% \begin{gather}
% y_i=z'_i\sin\phi_i, \quad z_i=z'_i\cos\phi_i\\
% v_{i}^y=v_{i}^{z'}\sin\phi_i, \quad  v_{i}^z=v_{i}^{z'}\cos\phi_i
% \end{gather}
% \label{eq:18}
% \end{subequations}
By leveraging the optimal control inputs, which include net thrust $ F_i $ and Euler angles $ (\alpha_i, \beta_i, \gamma_i) $, the coordinator can effectively guide the merging eVTOL $ i $ to execute the optimal take-off climb and merging trajectory with both efficiency and safety.
The optimal trajectory of merging eVTOL $i$ will be stored and may be used by $C_j$ to support the decision-making process for the subsequent merging eVTOLs in this TM section.

\section{Numerical experiments}
\label{sec: Numerical Experiments}

In this section, we present two case studies to evaluate the performance of the proposed airspace design and the HCTMM strategy. Case 1 considers a single corridor-vertiport pair. The corridor segment, aligned along the $X$-axis in the inertial coordinate system, has a height of $ L_c=305 \, \text{m}$. The vertiport is located at $ O_0=(0, 0, 0)$ and schedules take-off times at $t=[6,36,66,96,126,156]$ seconds. We validate the proposed strategy from the following three perspectives:
% \textit{(1) Safety assurances.} We first examine the safety constraint satisfaction of the proposed HCTMM strategy at different levels of traffic flow in the corridor. \textit{(2) Travel efficiency improvement: } We then compare the HCTMM strategy to the strategy that carries out autonomous take-offs with a fixed merging point (ATFM) using the first-come-first-served principle. 
% \textit{(3) Computational efficiency: } Finally, we compare the HCTMM strategy with a dynamic merging point strategy obtained through exhaustive search. 
\begin{itemize}
    \item \textit{Safety assurance:} We first examine the safety constraints satisfaction of the proposed HCTMM strategy at different levels of traffic flow in the corridor.
    \item \textit{Travel efficiency improvement:} We compare the HCTMM strategy to the strategy that carries out autonomous take-offs with a fixed merging point (ATFM) using the first-come-first-served principle.
    \item \textit{Computational efficiency:} We compare the HCTMM strategy with two alternative dynamic merging point strategies, one obtained through greedy-like exhaustive search (GES) and the other through rule-based exhaustive search (RES).
\end{itemize}
Case 2 examines the scalability of the HCTMM strategy. We consider the coordination between multiple corridor-vertiport pairs. This scenario includes a corridor segment (configured identically to Case 1) and four densely distributed vertiports $(O_1,\, O_2,\, O_3,\, O_4)$, with their locations $O_i$ and corresponding take-off demand $r_i$ listed in \autoref{tab:02}. Case 2 also validates the effectiveness of the proposed airspace design in improving computational efficiency.

We obtain the numerical solutions of \autoref{prob:6} and \autoref{prob:4} using the GPOPS-\uppercase\expandafter{\romannumeral2} toolbox based on MATLAB R2022b\footnote{ \hyperref[tab:03]{Table \ref*{tab:03}} in the appendix outlines the parameters adopted in the pseudo-spectral method in GPOPS-II.}. Following \cite{wu2022convex,EasyAccess2024}, other parameters are given in \autoref{tab:02}.

\begin{table}[h!]
    \centering
    \begin{threeparttable}[b]
    \caption{Parameters of the simulation environment}
    \begin{tabular}{llllll}
        \toprule
        % \hline
        Parameter\footnotemark[1] & Value & Unit & Parameter\footnotemark[1] & Value & Unit \\
        % \midrule
        % \hline
        \cmidrule(lr){1-3} \cmidrule(lr){4-6}
        $T_{\max}$ & 230 & $s$ & $g$ & 9.81 &$m/s^2$ \\
        $L_{o}$ & 600 & $m$ & $L_{m}$ & 1050 & $m$\\
        $L_{s}$ & 50 & $m$ & $L_{d}$ & 30.5 & $m$\\
        $s_{f}$ & 50 & $m$ & $T$ & 6 & $s$\\
        $v_{TOSS}$& 8 & $m/s$& $RoC_{\max}$& 9&$m/s$\\
        $\Theta_{dep}$&$2.58$& $\circ$& $\Delta  t$&0.2 &s\\
        $\lambda$ & 20 & - & $m_i$ & 240 & $kg$\\
        $F^{\max}_i$ & 4800 & $N$ & $V^{\max}_i$ & 40 & $m/s$\\ 
        $O_1$ & (0, 0, 0) & $m$ & $O_2$ & (700, 50, 0) & $m$\\
        $O_3$ & (700, -50, 0) & $m$ & $O_4$ & (800, 10, 20) & $m$\\
        $r_1$ & [11,46,81,116,151] & $s$ & $r_2$ & [16,56,96,136]  & $s$\\
        $r_3$ & [16,56,96,136] & $s$ & $r_4$ & [6,36,66,96,126,156]  & $s$\\
        \bottomrule
        % \hline
    \end{tabular}
    \begin{tablenotes}
     \item[1] Parameters with subscript $i$ are shared by all merging eVTOLs.
   \end{tablenotes}
   \label{tab:02}
   \end{threeparttable}  
\end{table}

\subsection{Case 1: A single corridor-vertiport pair under different corridor traffic flow levels}

In this case, the corridor traffic flow is modeled using a Bernoulli process. At each time step, a new eVTOL enters the observation zone with a probability $\rho$ governed by a Bernoulli distribution, provided that safety distance requirements are satisfied. The speed of corridor eVTOLs is randomly assigned within the range [17, 23]. At the initial time $ t_0 = 0 $, the horizontal positions of the eVTOLs in the observation zone are given by $[-50, -200, -320, -500, -600]$ m. During the simulation, the vertiport generates take-off requests for 6 merging eVTOLs (M1 to M6), while 39 eVTOLs in the light-traffic-flow scenario and 59 eVTOLs in the heavy-traffic-flow scenario pass through the observation zone.

\textit{(1) Safety assurances.} 
% Traffic managers implement the HCTMM strategy to ensure not only the flight safety of eVTOLs but also the efficiency of take-off and merging operations. 
Traffic managers implement the HCTMM strategy with one of the primary objectives being to ensure flight safety.
\hyperref[fig:A1]{Figure \ref*{fig:A1}} presents the horizontal position trajectory of all merging eVTOLs under the HCTMM strategy. The red dots represent the merging point of each merging eVTOL, and the blue shaded areas denote the collision avoidance area. It can be observed that the tactical conflict management level ensures all dynamic merging points keep safe separation distances from both preceding and following corridor eVTOLs. 
Meanwhile, at the operational trajectory optimization level, merging eVTOLs are guided to reach their dynamic merging points while matching the speed of preceding corridor eVTOLs and satisfying mechanical and safety constraints. For example, as shown in \hyperref[fig:position]{Figure \ref*{fig:position}}-\hyperref[fig:F1]{Figure \ref*{fig:F1}}, merging eVTOL M1 successfully follows an optimized trajectory to achieve smooth and constraint-compliant take-off climb and merging processes.
Furthermore, the complete merging process of all eVTOLs under both light and heavy corridor traffic flow conditions is presented in the supplementary video\footnote{The supplementary video is available at: \url{https://youtu.be/VN5Ji4HSmCE}.}, providing a more intuitive visualization of safety validation. 

\begin{figure}[h!]
    \centering
    \begin{subfigure}[b]{0.49\textwidth}
        \centering
        \includegraphics[width=1.15\textwidth]{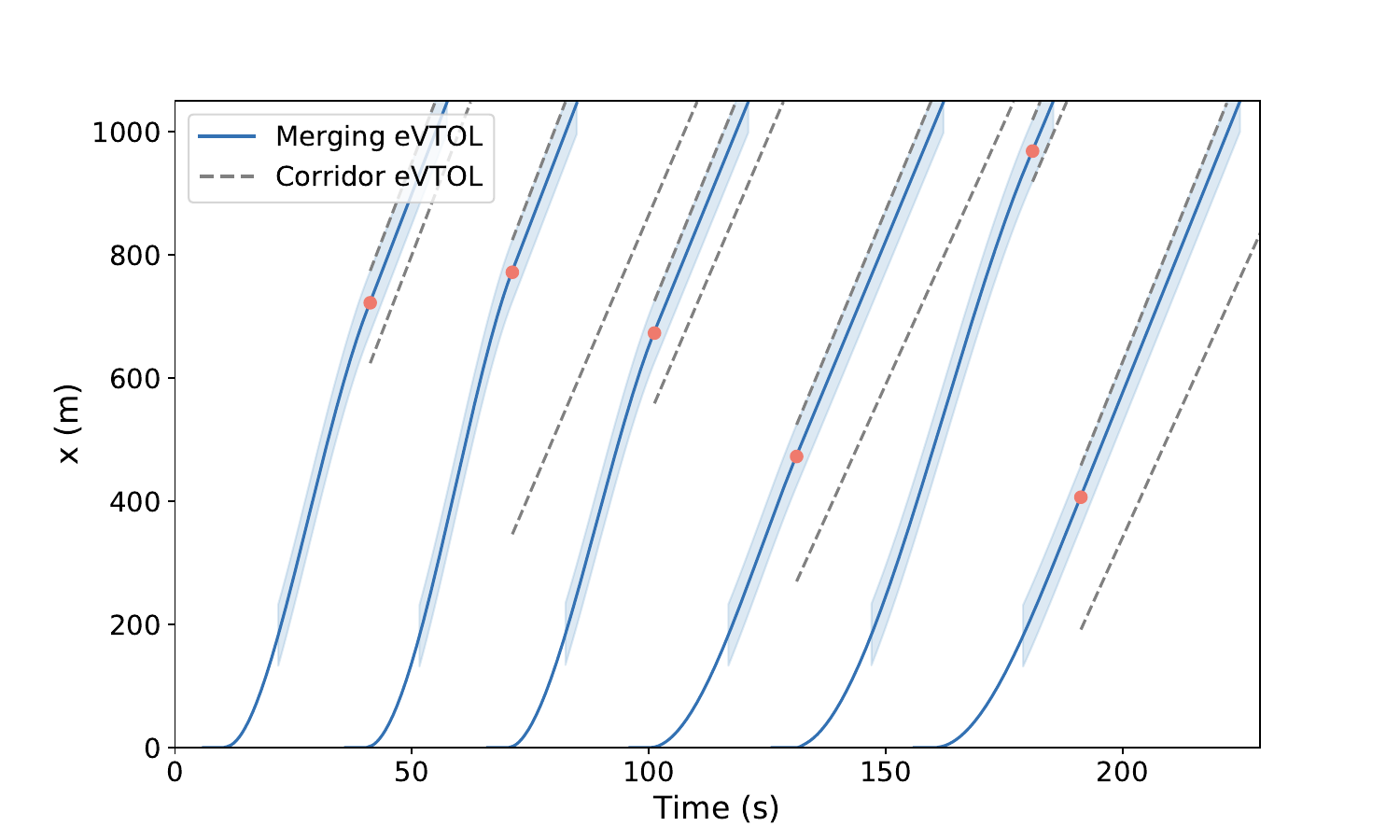}
        \caption{Light corridor traffic flow ($\rho=0.05$)}
        \label{fig:case1-ctdm}
    \end{subfigure}
    \hfill
    \begin{subfigure}[b]{0.49\textwidth}
        \centering
        \includegraphics[width=1.15\textwidth]{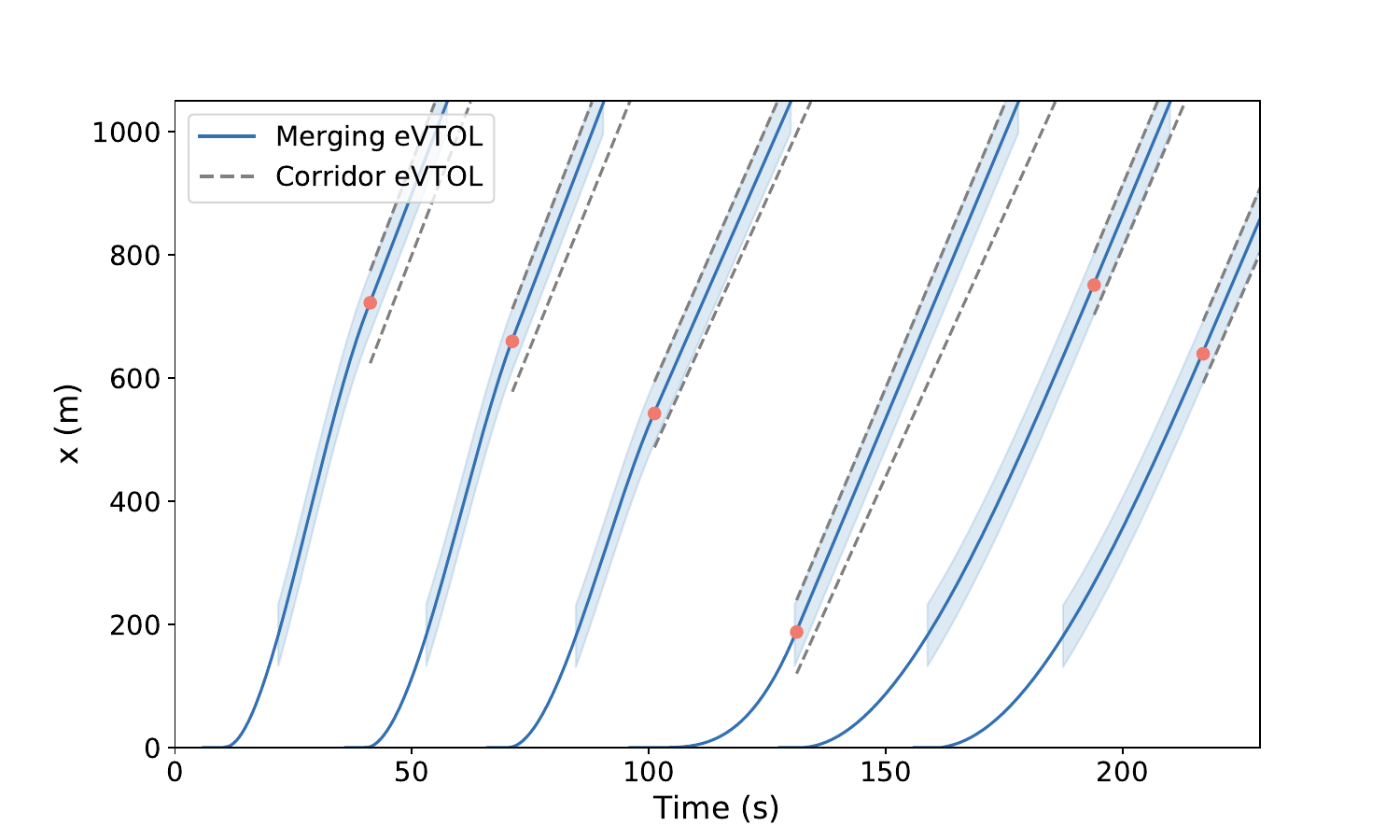}
        \caption{Heavy corridor traffic flow ($\rho=0.15$)}
        \label{fig:case2-ctdm}
    \end{subfigure}

    \caption{Horizontal position trajectory of merging eVTOLs under HCTMM strategy.}
    \label{fig:A1}
\end{figure}

\begin{figure}[h!]
    \centering
    \begin{subfigure}{0.49\textwidth}
        \centering \includegraphics[width=1.15\textwidth]{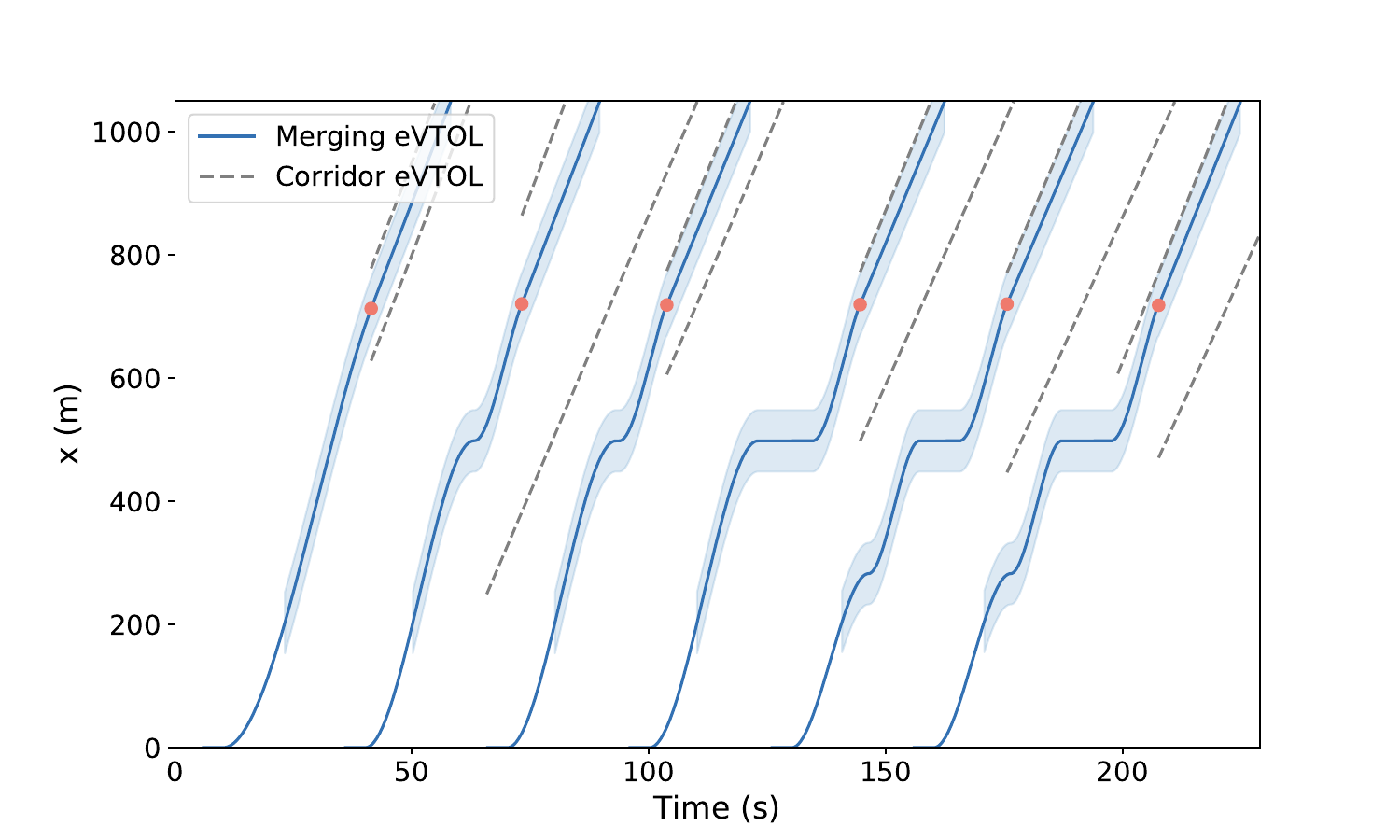}
        \caption{Light corridor traffic flow ($\rho=0.05$)}
    \end{subfigure}
    \hfill
    \begin{subfigure}{0.49\textwidth}
        \centering
        \includegraphics[width=1.15\textwidth]{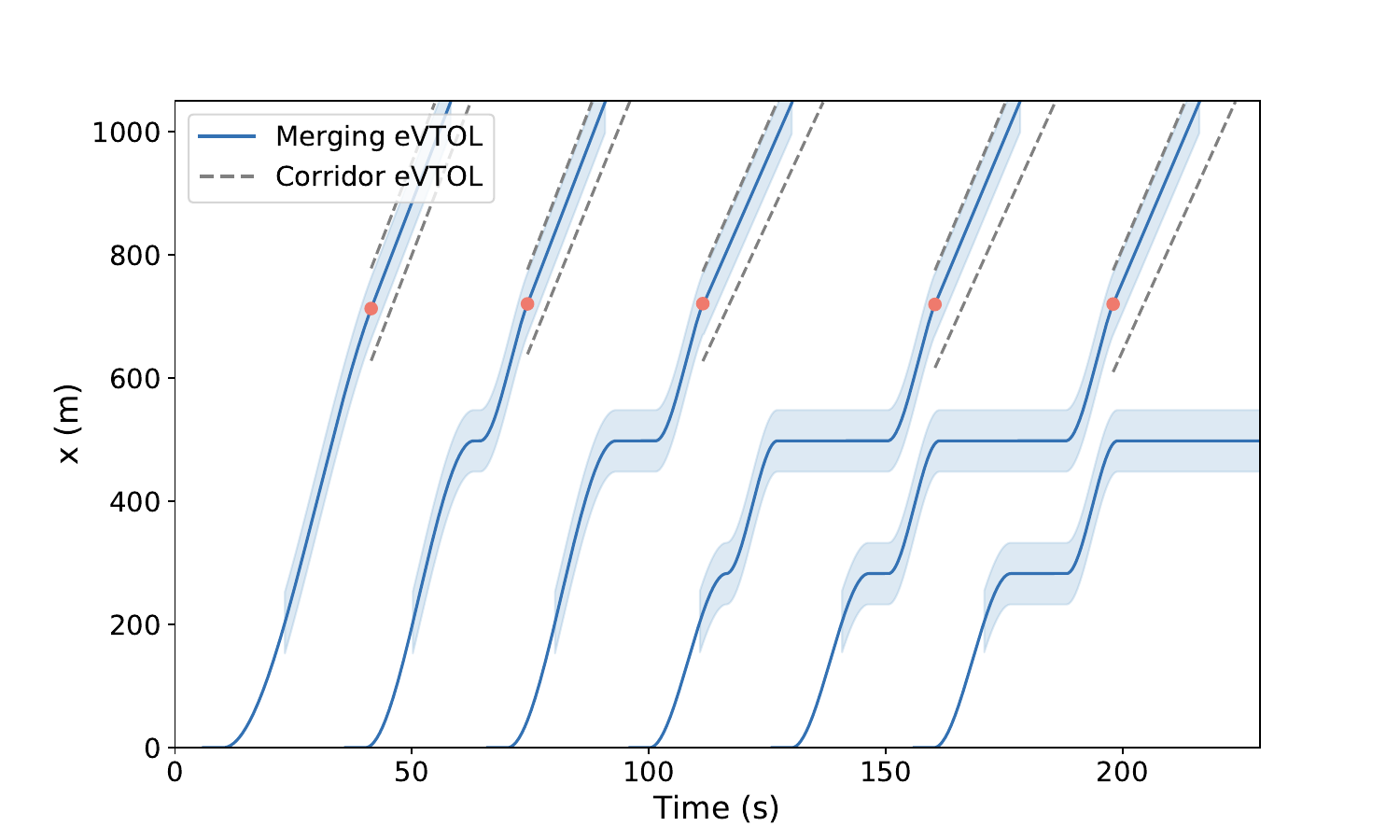} 
        \caption{Heavy corridor traffic flow ($\rho=0.15$)}
    \end{subfigure}

    \caption{Horizontal position trajectory of merging eVTOLs under ATFM strategy.}
    \label{fig:A3}
\end{figure}

\begin{figure}[h!]
    \centering
    \begin{subfigure}[t]{0.325\textwidth}
        \centering
        \includegraphics[width=0.95\textwidth]{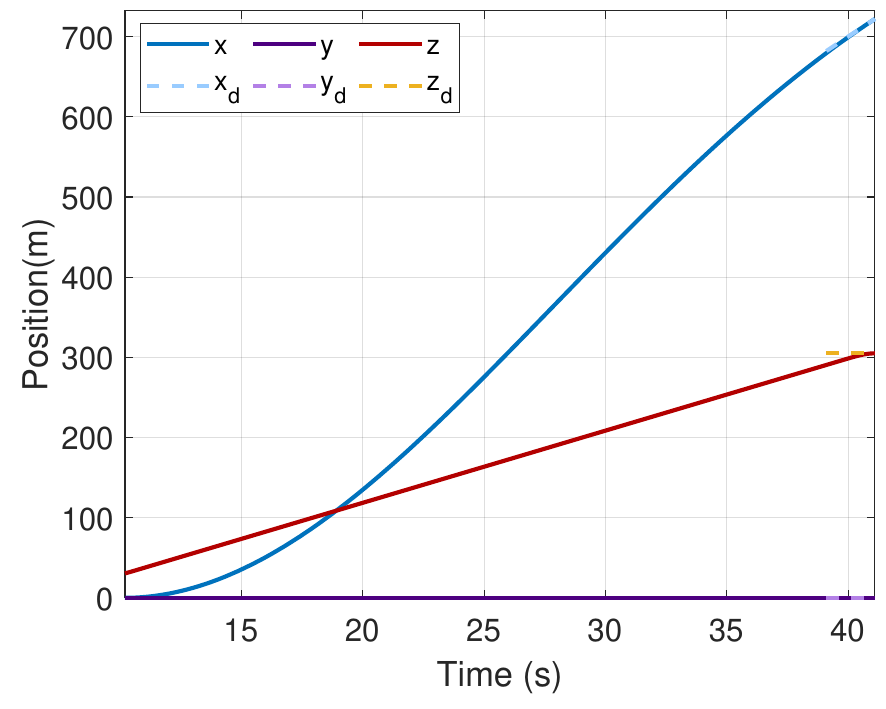}
        \caption{The position trajectory}
        \label{fig:position}
    \end{subfigure}
    \begin{subfigure}[t]{0.325\textwidth}
        \centering
        \includegraphics[width=0.95\textwidth]{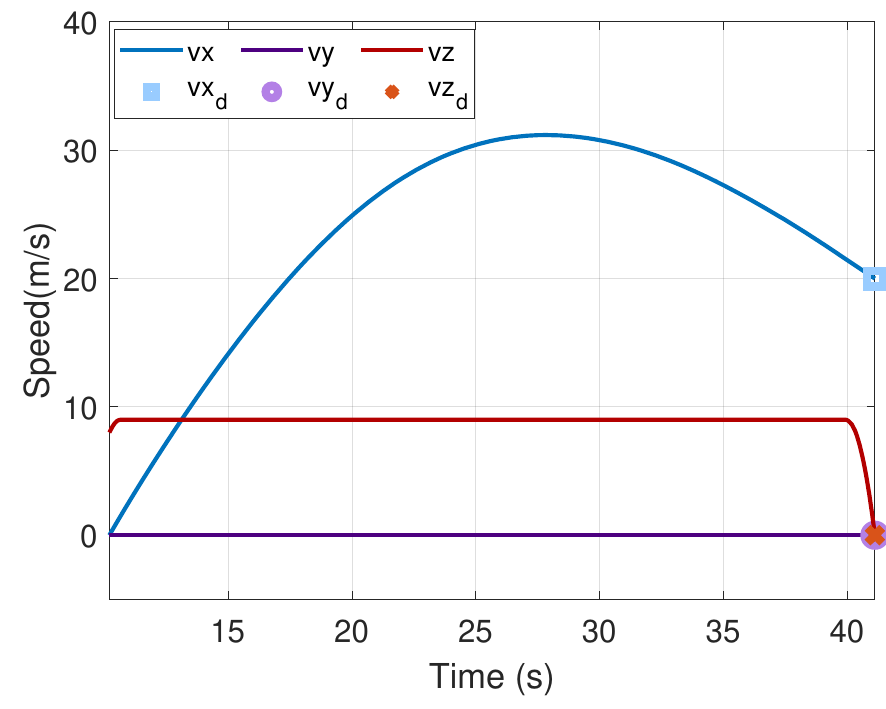}
        \caption{The speed trajectory}
        \label{fig:speed}
    \end{subfigure}
    \begin{subfigure}[t]{0.325\textwidth}
        \centering  \includegraphics[width=0.95\textwidth]{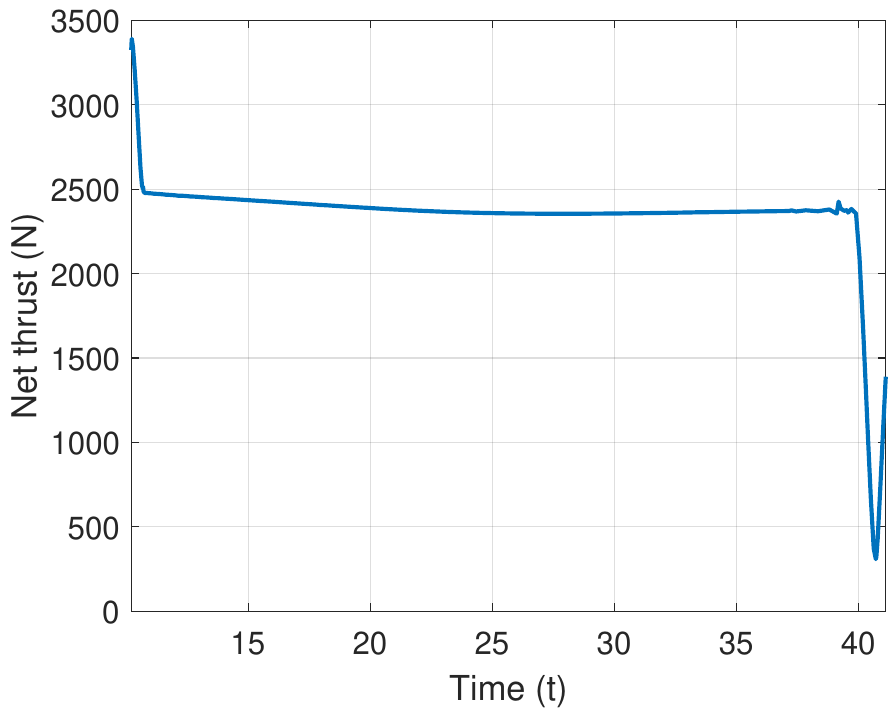}
        \caption{The net thrust input}
        \label{fig:F1}
    \end{subfigure}

    \caption{Simulation results of merging eVTOL M1 under the HCTMM strategy in light corridor traffic flow.}
    \label{fig:A2}
\end{figure}

\textit{(2) Travel efficiency improvement.} 
To validate the effectiveness of the tactical coordination of vertiports, we compare the proposed HCTMM strategy to a baseline take-off and merging strategy—autonomous take-off with a fixed merging point (ATFM). Selecting a fixed merging point is never a trivial task. If the merging point is too close to the observation zone, the climbing becomes steep, increasing energy consumption. If it is too close to the endpoint, there is insufficient longitudinal space for speed adjustment to align with the corridor flow. Based on these considerations, the merging point is designated at 720 m (approximately 68\% of the TM section length) to extend horizontal movement during the climb phase for energy efficiency, while retaining sufficient distance for speed adjustment and conflict resolution during merging. In the ATFM strategy, the trajectories comply with the path constraints specified in \cite{EasyAccess2024}. 
By analyzing both \hyperref[fig:A3]{Figure \ref*{fig:A3}} and the supplementary video, we observe that the majority of trajectories under the ATFM strategy include an airborne holding phase (sometimes even two), which significantly delays eVTOLs' merging completion times. In contrast, \hyperref[fig:A1]{Figure \ref*{fig:A1}} illustrates that the take-off and merging trajectories in the HCTMM strategy are significantly smoother. By managing take-off times and selecting dynamic merging points at the tactical level, the HCTMM strategy effectively mitigates potential conflicts. This prevents eVTOLs from airborne holding, which would otherwise increase energy consumption, disrupt subsequent flight missions, and reduce the capability for responding to incidents such as adverse weather conditions. 
\autoref{tab:05} presents the performance metrics for each merging eVTOL under both strategies. The HCTMM strategy exhibits shorter completion time and travel time, as compared to the ATFM strategy \footnote{Given that the HCTMM strategy employs variable merging points, the travel time, control cost and energy consumption — measured from the generation of demand to the leave from the TM section — serves as more representative metrics for evaluating the efficiency of take-off and merging strategies.}.
% In particular, under heavy traffic conditions, eVTOL M6 fails to complete the merging process within the total simulation duration under the ATFM strategy. 
Moreover, the HCTMM strategy effectively manages take-off times, achieving the conversion of airborne holding time into controlled take-off delays, resulting in a significant reduction in control cost for most eVTOLs.
% reduces merging times and control cost for most eVTOLs, with a more noticeable performance gap compared to the ATFM strategy under heavy traffic conditions. 
% While the ATFM strategy results in merging times exceeding the required intervals for eVTOLs M4, M5, and M6 fails to complete the merging process within the total simulation duration, the HCTMM strategy maintains stable and efficient performance. 

\hyperref[fig:A4]{Figure \ref*{fig:A4}} illustrates the performance benefits of the HCTMM strategy under different corridor traffic flow levels, compared to the ATFM strategy. 
To quantify the reduction in control cost for energy savings, we utilize a set of energy consumption models for different eVTOL operating states, as developed by \cite{borzemski2018information,ni2022energy,gong2023modeling}. As shown in \hyperref[fig:A4]{Figure \ref{fig:A4}}, the HCTMM strategy achieves substantial reductions in average travel time, control cost, and energy consumption, highlighting its superior performance. This advantage becomes increasingly pronounced as the corridor traffic volume increases. In particular,  under heavy traffic conditions ($ \rho = 0.15 $),  the ATFM strategy fails to complete the merging process within the total simulation duration $T_{\max}$, resulting in missing data. Its average travel time, control cost, and energy consumption would be significantly higher than those of the HCTMM strategy. These findings imply that in high-density, high-demand UAM environments, the HCTMM strategy will offer ongoing advantages in managing traffic flow.

% our tactical conflict management level avoids hovering delays during the take-off merging process by utilizing optimally selected merging points (see \autoref{fig:A1}). As a result, the proposed CTMD strategy significantly reduces the total travel time and energy consumption for most merging eVTOLs from take-off to exiting the TM section, as shown in \autoref{tab:04}. Specifically, the strategy achieves an average reduction of 10\% in travel time and 10\% in control cost.

% all merging eVTOLs successfully merge into the corridor traffic flow while maintaining safe separation distances from both preceding and following eVTOLs. 
% Take merging eVTOL 1 as an example, its position trajectory, speed trajectory and control input of the on-ramp vehicle are shown in \hyperref[fig:A2]{Figure \ref*{fig:A2}}.
% From \hyperref[fig:position]{Figure \ref*{fig:position}} - \hyperref[fig:speed]{Figure \ref*{fig:speed}}, we can find merging eVTOL 1 can exactly reach the dynamic merging point $(P_{i}^{d}(t_i^f))$ and match the speed $(v_d^x,v_d^y,v_d^z)$ of the preceding corridor eVTOL $d$ at the time $t_i^f$. 

\begin{table}[h!]%\small%\scriptsize
\centering
\begin{threeparttable}[b]
\caption{Performance comparison between HCTMM and ATFM strategies in Case 1. }
\begin{tabular}{c c cccc cccc}
    \toprule
    \multicolumn{2}{c}{} & \multicolumn{4}{c}{HCTMM strategy} & \multicolumn{4}{c}{ATFM strategy} \\
    \cmidrule(lr){3-6} \cmidrule(lr){7-10}
    \multicolumn{1}{c}{\parbox{0.9cm}{Traffic flow}} & \multicolumn{1}{c}{ID} & 
    \multicolumn{1}{c}{\parbox{1.3cm}{Take-off \\ time (s)}} & 
    \multicolumn{1}{c}{\parbox{1.6cm}{Completion \\ time (s)}} & 
    \multicolumn{1}{c}{\parbox{1.3cm}{Travel \\ time (s)}} & 
    \multicolumn{1}{c}{\parbox{1.3cm}{Control\\ cost (-)}} & 
    \multicolumn{1}{c}{\parbox{1.3cm}{Take-off \\ time (s)}} & 
    \multicolumn{1}{c}{\parbox{1.6cm}{Completion \\ time (s)}} & 
    \multicolumn{1}{c}{\parbox{1.3cm}{Travel \\ time (s)}} & 
    \multicolumn{1}{c}{\parbox{1.3cm}{Control\\ cost (-)}} \\
        \midrule
        \multirow{6}{*}{Light}& M1& 6.0&41.1&57.5&2609&6.0&41.3&58.0&2259\\
        & M2& 36.0&71.1&85.0&2499&36.0&73.2&89.9&2841\\
        & M3& 66.0&101.1&121.1&2779&66.0&103.8&121.4&2924\\
        & M4& 96.0&131.1&162.3&3294&96.0&144.6&162.5&3457\\
        & M5& 126.0&180.9&185.4&2940&126.0&175.6&193.9&3535\\
        & M6& 156.0&191.1&224.8&3407&156.0&207.6&225.0&3587\\
        \midrule
        \multirow{6}{*}{Heavy}& M1& 6.0&41.1&57.5&2609&6.0&41.3&58.0&2259\\
        & M2& 36.0&71.1&90.6&2751&36.0&74.4&91.1&2898\\
        & M3& 66.0&101.1&130.1&3194&66.0&111.4&130.2&3417\\
        & M4& 96.0&131.1&178.1&4050&96.0&160.4&178.4&4264\\
        & M5& 127.6&193.9&210.1&4028&126.0&198&216.3&4644\\
        & M6& 156.0&216.8&239.7&4093&156.0&-&-&-\\
        \bottomrule
    \end{tabular}
    \label{tab:05}
  \end{threeparttable}
\end{table}

\begin{figure}[h!]
    \centering
    \begin{subfigure}{0.49\textwidth}
        \centering \includegraphics[width=1.05\textwidth]{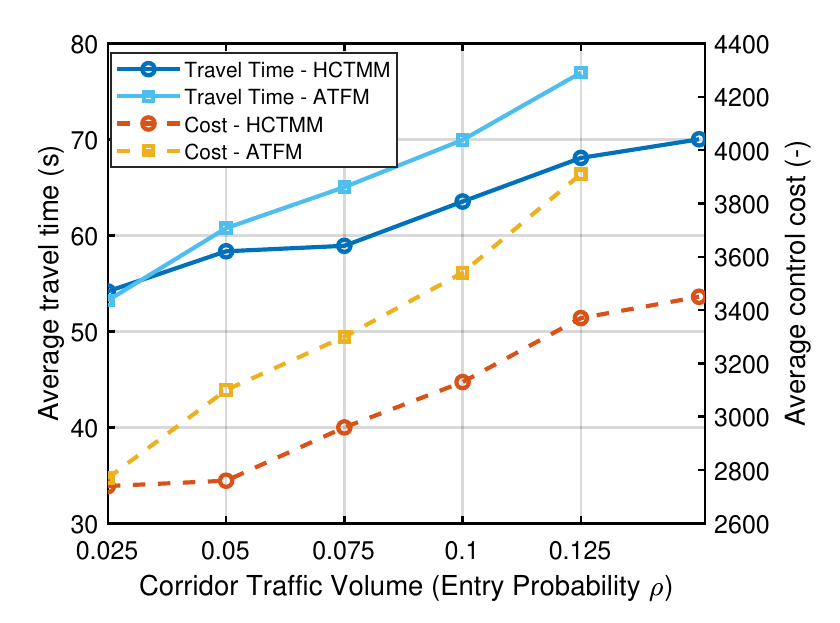}
        \caption{ Control cost}
    \end{subfigure}
    \hfill
    \begin{subfigure}{0.49\textwidth}
        \centering
        \vspace{-0.75em} 
        \includegraphics[width=1.05\textwidth]{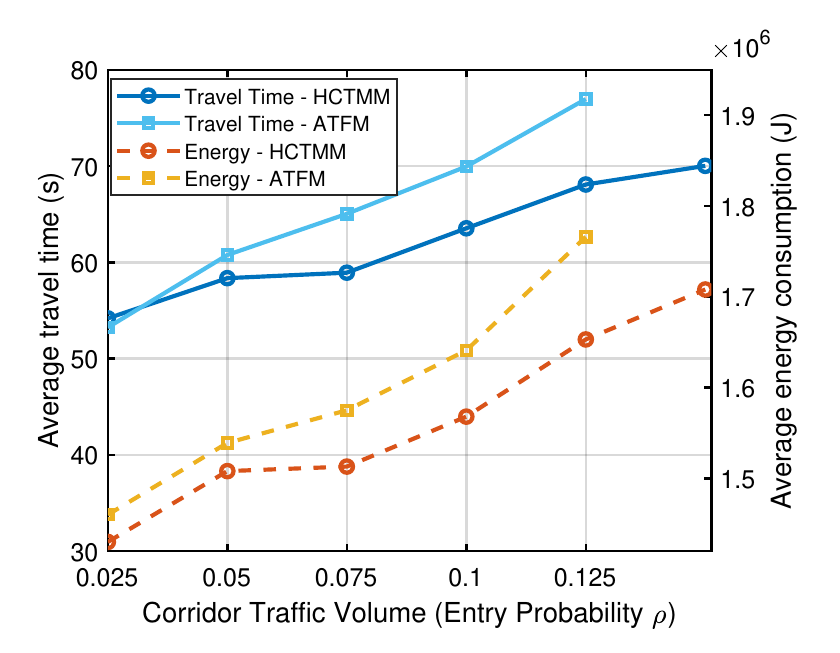} 
        \caption{Energy consumption}
    \end{subfigure}
    \caption{Performance comparison of both strategies in cases of different levels of corridor traffic flow.}
    \label{fig:A4}
\end{figure}

% \textcolor{red}{
\textit{(3) Computational efficiency.}
Computational efficiency is essential for real-time air traffic management. At the tactical level, we propose a heuristic algorithm for selecting dynamic merging points. To the best of our knowledge, this is the first work that addresses dynamic merging point selection in the context of UAM corridor merging control. In this part, we compare the proposed HCTMM strategy with two take-off and merging strategies based on dynamic merging points: greedy-like exhaustive search (GES) and rule-based exhaustive search (RES).
\begin{itemize}
    \item The GES strategy exhaustively searches all candidate merging pairs for each merging eVTOL and plans the corresponding take-off and merging trajectories by following a procedure similar to that at the operational level (see \autoref{sec: Operational deconfliction control}). However, due to the absence of feasibility and safety guarantees in this strategy, one cannot ensure that the OCP associated with every candidate merging pair is feasible. To address this issue, part of the safety constraints in  \autoref{prob:4} is shifted to a post-solution safety filter. Among the candidate merging pairs that pass the safety filter, the GES strategy selects the one with the lowest total cost as the per-target merging pair. The total cost is computed as a convex combination of the total travel time and control cost.
    \item In the RES strategy, all candidate merging pairs are first checked for safety and validity at the requested take-off time, following the procedure described in \autoref{sec: Tactical conflict management}. For each safe and valid pair, the optimal take-off and merging trajectory is planned using the same method as at the operational level (see \autoref{sec: Operational deconfliction control}) for every merging eVTOL. Among the resulting trajectories, the merging pair with the lowest total cost is selected as optimal, and its corresponding merging point is designated as the target dynamic merging point. 
\end{itemize}
% }

% \textcolor{red}{
\autoref{tab:06} shows the per-target merging pair selection and computation time of the HCTMM, GES, and RES strategies at light and heavy traffic scenarios. All three strategies yield the same merging pairs under the light traffic condition. Under the heavy traffic condition, the merging pairs by the HCTMM and RES are slightly different from those by the GES. Some of the take-off times determined by the GES are later than those by the HCTMM and RES. In addition, the GES selects some per-target merging pairs later than the other two strategies.
As discussed in \autoref{sec:4.1}, the distance between candidate merging pairs changes dynamically over time. Due to the lack of a comprehensive assessment of the safety and validity of each candidate merging pair throughout the entire TM section, the GES strategy exhibits a short-sighted decision-making process. Specifically, it may generate trajectories with potential collision risks for merging pairs that appear unsafe in the early phase but become safe and valid later. The GES strategy may also mistakenly select pairs that seem safe and valid at the beginning but later violate the safety principle. Consequently, such candidate merging pairs are typically rejected by the safety filter or may even lead to delayed take-off, as indicated by \autoref{tab:06}. 
% In contrast, both our strategy and the RES strategy overcome the short-sightedness of GES by incorporating a comprehensive safety and validity assessment. 
Moreover, the HCTMM produces the same results as the RES strategy in terms of merging pair selection under both traffic conditions.
However, the computation time of the HCTMM is much less than that of the RES and GES. Nevertheless, the average computation time per eVTOL under RES and GES strategies exceeds the reserved planning time threshold of $T = 6$ seconds, making them intractable for real-time implementation.
These findings indicate that the HCTMM can identify optimal merging points whilst satisfying real-time operational requirements.
% }

% \begin{table}[h!]%\small%\scriptsize
% \centering
% \begin{threeparttable}[b]
% \caption{Performance comparison between HCTMM, GES and RES strategies in Case 1. }
% \begin{tabular}{cccc cccc}
%     \toprule
%     \parbox{0.9cm}{\centering Traffic flow}
%     & \parbox{2.2cm}{\centering Take-off time [HCTMM and RES](s) } & \parbox{1.5cm}{\centering Take-off time [GES](s) }&
%     \parbox{3.2cm}{\centering Target merging pair selected [HCTMM and RES]} & \parbox{2.8cm}{\centering Target merging pair selected [GES]}&
%     \parbox{1cm}{\centering HCTMM\\  comp.\\(s)} & 
%     \parbox{0.95cm}{\centering RES\\ comp.\\(s)}&
%     \parbox{0.95cm}{\centering GES\\ comp.\\(s)}\\
%     \midrule
%      light&\makecell{6.0,36.0\\66.0,96.0\\126.0,156.0}&\makecell{6.0,36.0\\66.0,96.0\\126.0,156.0} &\makecell{(h1,h2)(h5,h6)\\(h9,h10)(h16,h17)\\(h19,h20)(h26,h27)}&\makecell{(h1,h2)(h5,h6)\\(h9,h10)(h16,h17)\\(h19,h20)(h26,h27)}&4.42&9.67&15.94\\
    
%      Heavy&\makecell{6.0,36.0\\66.0,96.0\\127.6,156.0}&\makecell{6.0,36.0\\66.0,96.2\\132.8,188.2} &\makecell{(h1,h2)(h7,h8)\\(h14,h15)(h26,h27)\\(h34,h35)(h40,h41)}&\makecell{(h1,h2)(h7,h8)\\(h14,h15)(h26,h27)\\(h34,h35)(h40,h41)}&5.83&10.84&68.80\\
%         \bottomrule
%     \end{tabular}
%     \label{tab:06}
%   \end{threeparttable}
% \end{table}

\begin{table}[h!]
\centering
\begin{threeparttable}[b]
\caption{Performance comparison between HCTMM, GES and RES strategies in Case 1.}
\renewcommand{\arraystretch}{1.1}
\begin{tabular}{>{\centering\arraybackslash}m{4.8cm} c c}
    \toprule
    Index& Light traffic flow & Heavy traffic flow \\
    \midrule
    \makecell[c]{Take-off time by HCTMM 
    \\ and RES (s)}
    & \makecell{6.0, 36.0, 66.0\\96.0, 126.0, 156.0} 
    & \makecell{6.0, 36.0, 66.0\\96.0, 127.6, 156.0} \\

    \makecell[c]{Take-off time by GES (s)}
    & \makecell{6.0, 36.0, 66.0\\96.0, 126.0, 156.0} 
    & \makecell{6.0, 36.0, 66.0\\96.2, 132.8, 188.2} \\

    \makecell[c]{Per-target merging pair selected \\ by HCTMM and RES}
    & \makecell{(h1,h2)(h5,h6)(h9,h10)\\(h16,h17)(h19,h20)(h26,h27)} 
    & \makecell{(h1,h2)(h7,h8)(h14,h15)\\(h26,h27)(h34,h35)(h40,h41)} \\

    \makecell[c]{Per-target merging pair selected \\by GES}
    & \makecell{(h1,h2)(h5,h6)(h9,h10)\\(h16,h17)(h19,h20)(h26,h27)} 
    & \makecell{(h1,h2)(h7,h8)(h14,h15)\\(h26,h27)(h37,h38)(h47,h48)} \\

    \makecell[c]{HCTMM comp. (s)} & 4.42 & 5.83 \\
    \makecell[c]{RES comp. (s)}   & 9.67 & 10.84 \\
    \makecell[c]{GES comp. (s)}   & 15.94 & 68.80 \\
    \bottomrule
\end{tabular}
\label{tab:06}
\end{threeparttable}
\end{table}

\subsection{Case 2: Coordination between multiple corridor-vertiport pairs}

In this case, we simulate the merging scenario of denser multiple corridor-vertiport pairs. The corridor traffic flow is generated by the same routine as in Case 1. At the initial time, the horizontal positions of the eVTOLs in the observation zone are given by $[600,500,380,300,200,150,-50, -250, -300, -500, -600]$ m. 
During the simulation, the vertiport will generate take-off requests for 19 merging eVTOLs and 61 corridor eVTOLs will pass through the corridor segment. 

How to ensure coordination between take-off and merging trajectories in densely distributed vertiports will be a key concern for future airspace managers. As shown in \hyperref[fig:A6]{Figure \ref{fig:A6}} and the supplementary video\footnote{A supplementary video is available at: \url{https://youtube.com/shorts/kW26AHOI5ig}}, all merging eVTOLs smoothly and successfully merge into the corridor. Moreover, merging eVTOLs with close take-off times and located at nearby vertiports can maintain safe separation throughout the entire process. This implies that the HCTMM strategy can be effectively extended to scenarios involving multiple corridor-vertiport pairs. By integrating the take-off airspace design and the proposed HCTMM strategy, the spatial and temporal separation of the eVTOLs is achieved, ensuring strict safety guarantees.
\autoref{fig:11} demonstrates the effectiveness of the proposed take-off airspace design in reducing the computational time under different corridor traffic levels. 
Without the TM section design, the average computation time per eVTOL increases with the increase in corridor traffic volume, while such performance metric is not significantly influenced by the increase in corridor traffic volume if we conduct the TM section design.
For all corridor flow levels, the TM section design keeps the average computation time below the reserved planning time threshold of $T=6$ seconds.
%Overall, the TM section design can reduce the average computation time by more than half of the original one.
This is one of the merits of the TM section design, which reduces the number of variables and obstacle avoidance constraints, effectively alleviating the computation time to between one-half and one-third of that under the free take-off airspace. This enhances the real-time feasibility of the HCTMM strategy for coordinated take-off merging management across multiple corridor-vertiport pairs.

\begin{figure}[htbp]
    \centering
    \begin{minipage}{0.49\textwidth}
        \centering
        \includegraphics[trim=0 170 0 100, clip, width=1.05\textwidth]{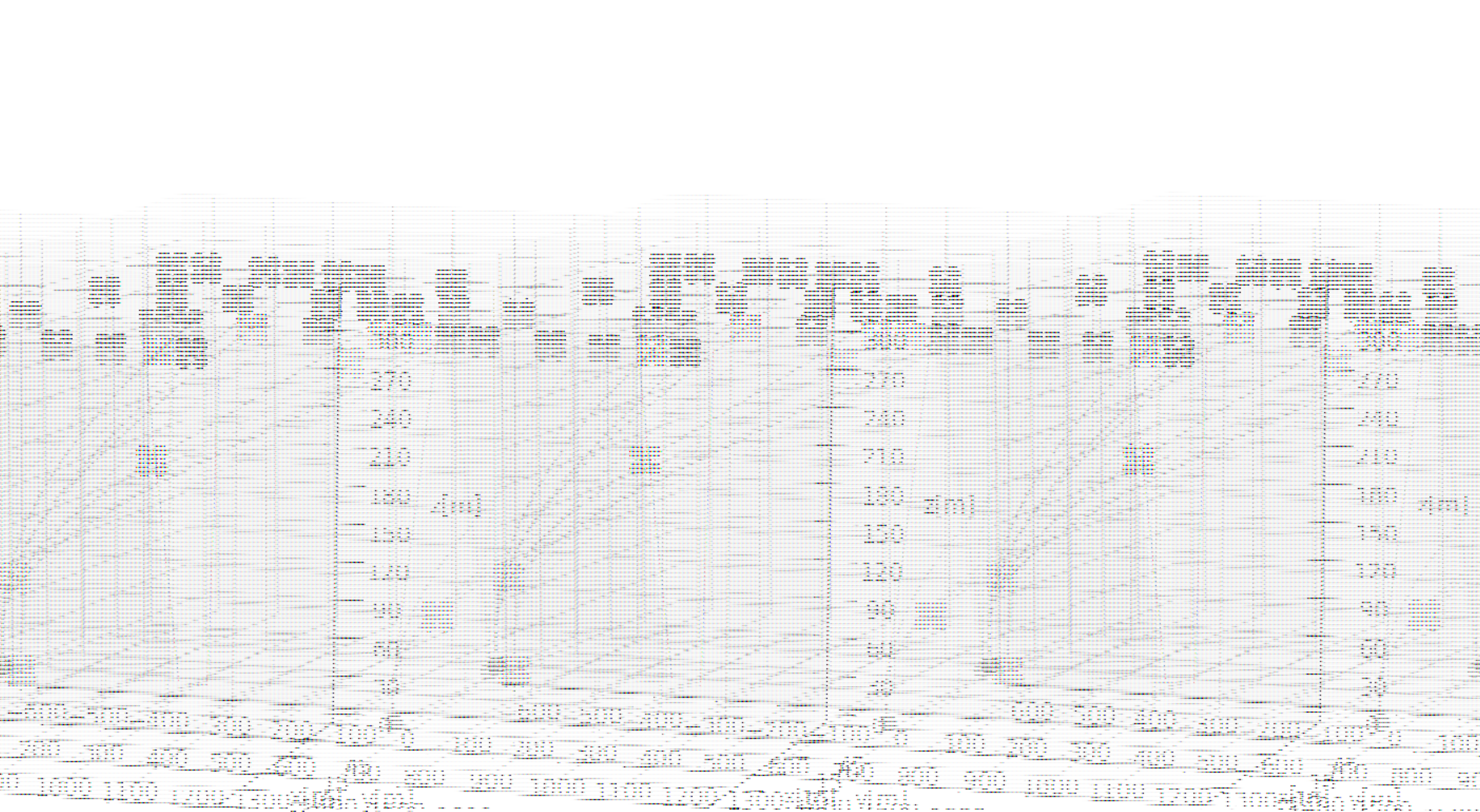} 
        \caption*{(a) $t=66s$}
    \end{minipage}
    \hfill
    \begin{minipage}{0.49\textwidth}
        \centering
        \includegraphics[trim=0 170 0 100, clip,width=1.05\textwidth]{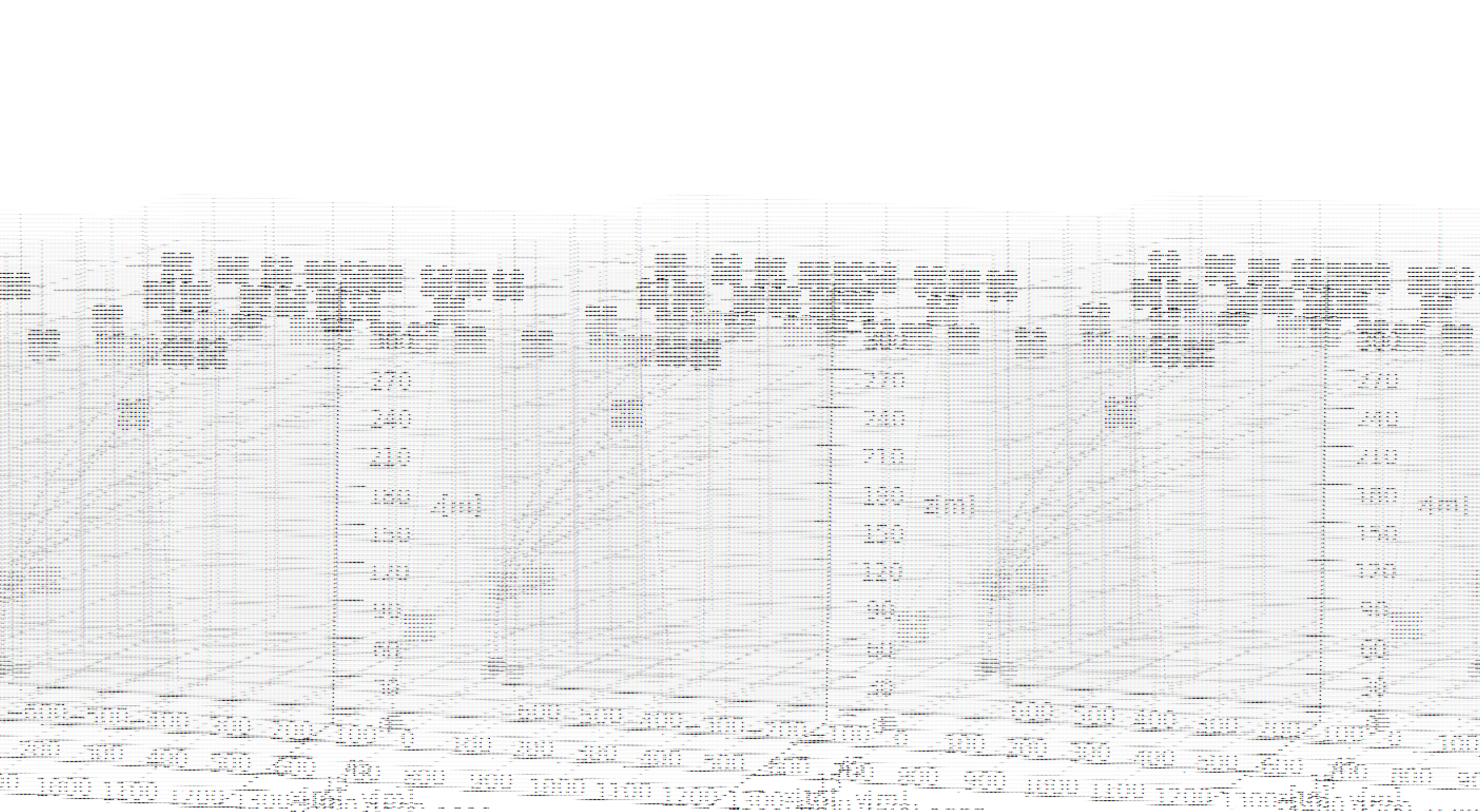} 
        \caption*{(b) $t=105s$}
    \end{minipage}

    \begin{minipage}{0.49\textwidth}
        \centering
        \includegraphics[trim=0 170 0 100, clip,width=1.05\textwidth]{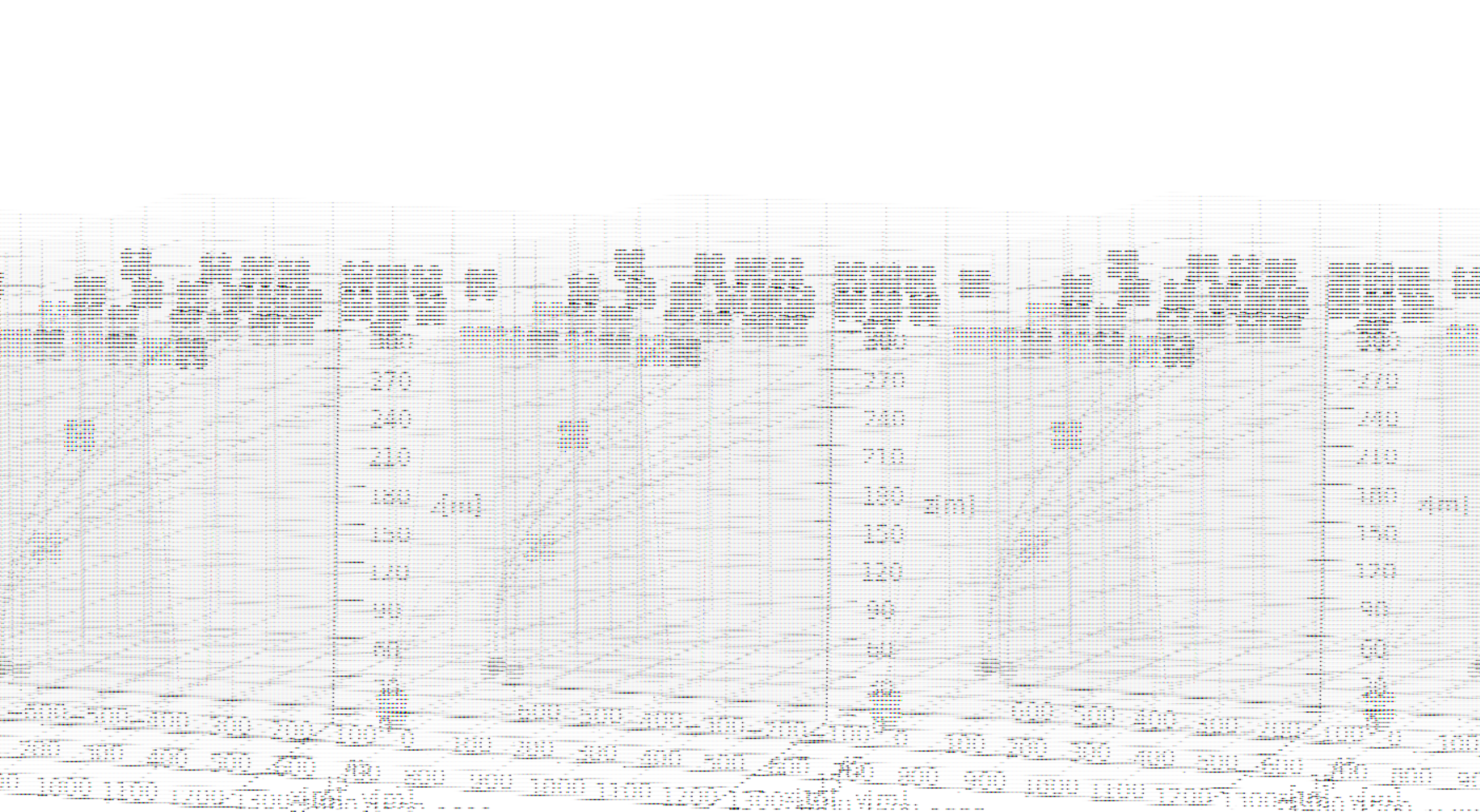} 
        \caption*{(c) $t=138s$}
    \end{minipage}
    \hfill
    \begin{minipage}{0.49\textwidth}
        \centering
        \includegraphics[trim=0 170 0 100, clip,width=1.05\textwidth]{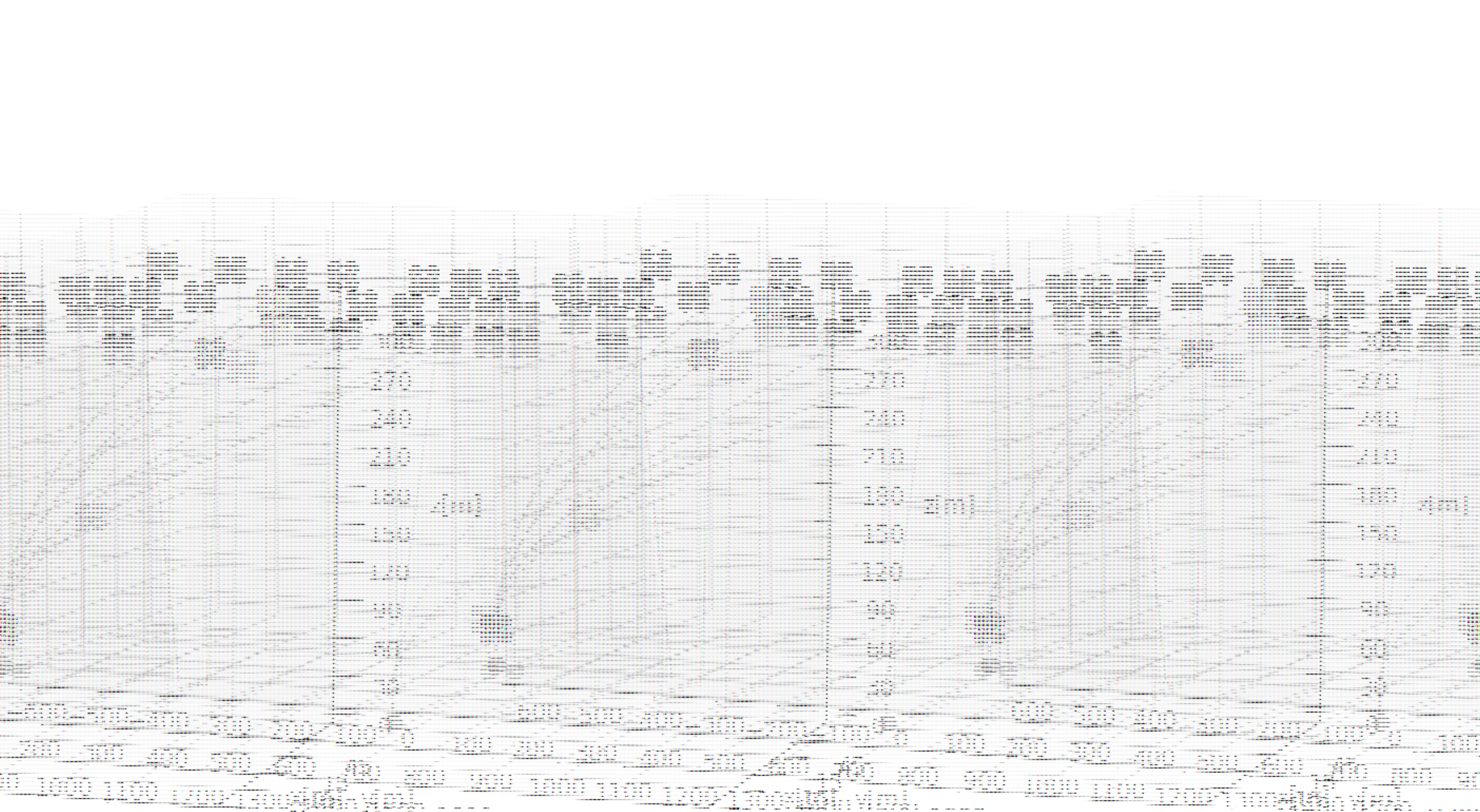} 
        \caption*{(d) $t=155s$}
    \end{minipage}

    % \vspace{0.0em} 
    \caption{Key frame from the simulation video for Case 2 ($\rho = 0.10$), where the width of a single eVTOL symbol approximates the safe distance.}
    \label{fig:A6}
\end{figure}

\begin{figure}[h]
    \centering
    \includegraphics[width=0.46\textwidth]{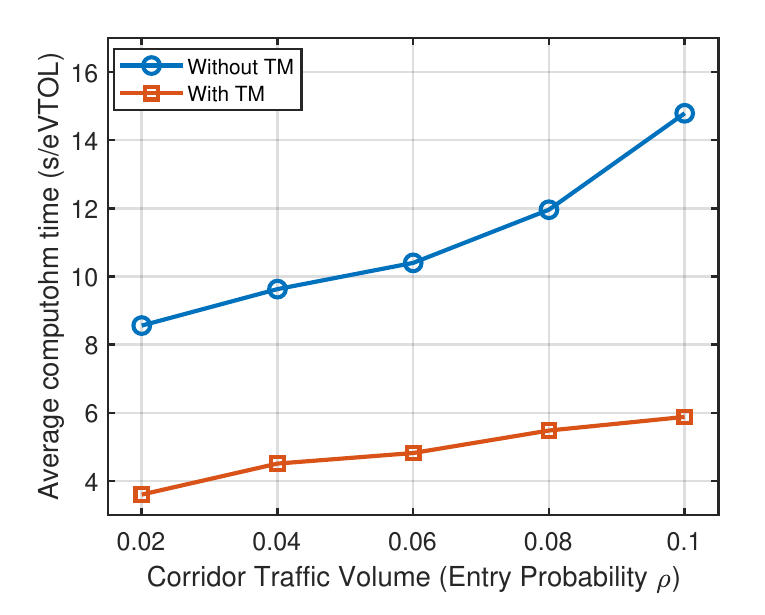}
    \caption{Computational efficiency analysis of HCTMM for take-off airspace design in Case 2.}
    \label{fig:11}
\end{figure}

\section{Conclusion}
\label{sec:Conclusion}
This paper presents an integrated strategy for managing take-off and merging operations in Urban Air Mobility (UAM) corridors. The proposed structured design for take-off airspace can significantly reduce trajectory planning complexity and simplify obstacle avoidance maneuvers. 
The hierarchical coordinated take-off and merging management (HCTMM) strategy integrates tactical scheduling and operational trajectory optimization. At the tactical level, the scheduling algorithm coordinates aircraft take-off times while assigning dynamic merging points to mitigate potential conflicts and minimize airborne holding. At the operational level, the trajectory optimization framework minimizes control cost and flight time subject to safety constraints.
Simulation results validate the effectiveness of our approach under various traffic conditions. The proposed strategy can reduce flight time, control cost, and enhance computational efficiency. This work establishes a framework to support seamless and conflict-free merging of electric vertical take-off and landing (eVTOL) aircraft into urban airspace corridors. 
Future efforts will be devoted to more complex scenarios, such as more complex corridor structures (e.g., multiple lanes). Additionally, further research will explore the collaborative control of corridor eVTOLs to achieve a trade-off between the take-off delay of merging eVTOLs and the additional delays introduced by coordinated merging within the corridor.
\begin{appendices}

% \section*{Appendix}
% \subsection*{A.1. Key notations}
\section{Key notations}
  \hyperref[tab:01]{Table \ref*{tab:01}} collects
the key notations throughout this paper.

\begin{table}%[!h]
    \centering
    \caption{Key notations}
    \begin{tabular}{p{0.7in}p{5.5in}}
        \toprule
        Notation & Description \\
        \midrule
        $i $ & Index of merging eVTOL from the set $\mathcal{I}$, comprising eVTOLs taking-off in the considered region\\
        $O_i$ & The vertiport position where eVTOL $i$ takes off, $O_i=(O_i^x,O_i^y,O_i^{z})$\\
        $D_i$ & The transition point position associated with vertiport $O_i$, $D_i=(O_i^x,O_i^y,O_i^{z}+L_d)$\\
        $V_{TOSS}$ & The take-off safety speed, with its magnitude denoted by $v_{TOSS}$ \\
        $\phi_i$ & Rotational angle from the $X_i^*Y_i^*Z_i^*$ frame to the $X_i'Y_i'Z_i'$ frame\\ 
        $H(t)$ & The set of eVTOLs within the observation zone of the corridor at time $t$\\
        $H_{e}(t)$ & The extended set of eVTOLs within the observation zone of the corridor at time $t$\\
        $N_{H}(t)$ & The number of corridor eVTOL set $H(t)$ at time $t$ \\
        $N_{H_e}(t)$ & The number of extended corridor eVTOL set $H_{e}(t)$ at time $t$ \\
        $p_{h_k}(t)$ & The position of corridor eVTOL $h_k$ at time $t$, $p_{h_k}(t)=\left(x'_{h_k}(t),z'_{h_k}(t)\right)$\\
        $V_{h_k}$ & Constant speed vector of corridor eVTOL $h_k$ in the observation zone and TM section, $V_{h_k}=(v_{h_k}^{x'},v_{h_k}^{z'})$\\
        $L_{o}$ & The length of observation zone\\
        $L_{c}$ & The height of corridor\\
        $L_{m}$ & The width of the TM section\\
        $L_{s}$ & The minimum safe distance between the adjacent corridor eVTOLs\\
        $s_f$ & The safe distance between two eVTOLs\\
        $t_{i}^{q}$ & The requested take-off time of the merging eVTOL $i$\\
        $t_{i}^{0}$ & The planned take-off time of the merging eVTOL $i$\\  
        $\mathcal{J}$& The set of eVTOLs in the same TM section as eVTOL $i$ that have not merged at time $t_i^0$\\
        $C_j$ & The trajectory of eVTOL $j$ with potential collision risk involving eVTOL $i$, where $j \in \mathcal{J}$\\
        $p^{ter}$ & The position of the endpoint of TM section for vertiport $O_i$, $p^{ter}=(L_{m},(L_c-L_d) \sec\phi_i)$\\
        $T_{k}^{ter}$ & The time for eVTOL $h_k$ to reach the endpoint $(ter)$ of TM section \\
        $\lambda$ & Time-to-energy value conversion coefficient\\
        $m_i$ & The mass of eVTOL $i$ \\
        \midrule
        $P_{i}^{d}(t)$ & The target dynamic position of the merging point of the merging eVTOL $i$ at time $t$, $P_{i}^{d} (t)=\left({x'}^d_i(t), (L_c-L_d) \sec\phi_i \right)$ \\
        $ \bar{t}_{ki}^{a}$ & The shortest time for eVTOL $i$ to reach the farthest merging point $p^a$ for pair $(h_k,\ h_{k+1})$\\
        $\bar{t}_{i}^{0}$ & The actual take-off time of the merging eVTOL $i$\\
         $\bar{t}_{i}^{v}$ & The actual take-off climb initiation time of merging eVTOL $i$\\
        $t_{i}^{f}$ & The merging completion time of the merging eVTOL $i$\\
        $x'_i(t)$, $z'_i(t)$ & The horizontal position, vertical position of the merging eVTOL $i$ at time $t$\\
        $ v_{i}^{x'}(t)$, $ v_{i}^{z'}(t)$ & The horizontal speed, vertical speed of the merging eVTOL $i$ at time $t$ \\
        $F_{i}(t)$ & The net thrust of eVTOL $i$ at time $t$\\
        $\theta_{i}(t)$ & The tip-path-plane pitch angle of eVTOL $i$ at time $t$\\
        \bottomrule
    \end{tabular}
    \label{tab:01}
\end{table}

% \subsection*{A.2. Simplification of Dynamical Systems}
\section{Simplification of Dynamical Systems}
\label{sec: Appendix}
As the take-off climb and merging trajectory optimization problem in free-flight airspace, nonlinear dynamics with 4-dimensional control variables, i.e., the net thrust $ F_i $ and Euler angles $ (\alpha_i, \beta_i, \gamma_i) $, and the 6-dimensional state variables, i.e., the position $(x_i, y_i, z_i) $, and the speed $(v_{i}^x, v_{i}^y, v_{i}^z) $ are considered \citep{quan2017introduction}. The motion dynamics of eVTOL in free-flight airspace are as follows:
\begin{subequations}
\label{eq:b}
\begin{align}
\dot{x_i}&=v_{i}^x \label{eq:b1}\\
\dot{y_i}&=v_{i}^y \label{eq:b2}\\
\dot{z_i}&=v_{i}^z \label{eq:b3}\\
\dot{v}_{i}^x&=\frac{F_i}{m_i}(\cos\alpha_i \sin\beta_i \cos\gamma_i+ \sin\alpha_i \sin \gamma_i)\label{eq:b4}\\
\dot{v}_{i}^y&=\frac{F_i}{m_i}(\cos\alpha_i \sin\beta_i \sin\gamma_i - \sin\alpha_i \cos \gamma_i)\label{eq:b5}\\
\dot{v}_{i}^z&=\frac{F_i}{m_i}\cos\beta_i \cos\alpha_i-g\label{eq:b6}
\end{align}
\end{subequations}
where $ g $ denotes the gravitational acceleration constant and $m_i$ represents the mass of eVTOL $i$. \eqref{eq:b4}-\eqref{eq:b6} can be simplified as
\begin{equation}
\label{eq:g}
    \dot{\bm{v}}=\frac{F_i}{m_i} R_{1,i} e_3-g e_3
\end{equation}
where $\bm{v}=[v_i^x\;v_i^y\;v_i^z]^T$, $ e_3 = [0\;0\;1]^T$. The rotation matrices $ R_{1,i} $ is defined as follows:
\begin{equation}
\label{eq:e}
    R_{1,i}=\begin{bmatrix}
 \cos\beta_i \cos\gamma_i& \sin\alpha_i \sin\beta_i \cos\gamma_i  - \cos\alpha_i \sin \gamma_i & \cos\alpha_i \sin\beta_i \cos\gamma_i + \sin\alpha_i \sin \gamma_i \\
 \cos\beta_i \sin\gamma_i&  \sin\alpha_i \sin\beta_i \sin\gamma_i + \cos\alpha_i \cos \gamma_i & \cos\alpha_i \sin\beta_i \sin\gamma_i - \sin\alpha_i \cos \gamma_i\\
 -\sin\beta_i& \sin\alpha_i \cos\beta_i& \cos\beta_i \cos\alpha_i
\end{bmatrix}
\end{equation}
Building on the structured take-off airspace, we constrain the eVTOL's trajectory to the $X_i'D_iZ_i'$ plane, which is defined by the $X_i'$-axis, $Z_i'$-axis, and the origin $D_i$ of the $X_i'Y_i'Z_i'$ coordinate system. The $X_i'Y_i'Z_i'$ coordinate system is constructed by first translating the inertial frame $X_iY_iZ_i$ such that its origin shifts from the vertiport $O_i$ to the corresponding transition point $D_i$ for each eVTOL $i$, resulting in an intermediate frame $X_i^*Y_i^*Z_i^*$. A subsequent clockwise rotation by an angle $\phi_i$ about the $ X_i^*$-axis is then applied to this intermediate frame. The coordinate transformation procedure is illustrated in \autoref{fig:12}. The associated rotation matrix $R_{2,i}$ is given by:
\begin{equation}
\label{eq:f}
    R_{2,i}=\begin{bmatrix}
 1&0&0\\
 0&\cos\phi_i&-\sin\phi_i\\
 0& \sin\phi_i& \cos\phi_i
\end{bmatrix}
\end{equation}
\begin{figure}[h]
    \centering
    \includegraphics[width=0.55\textwidth]{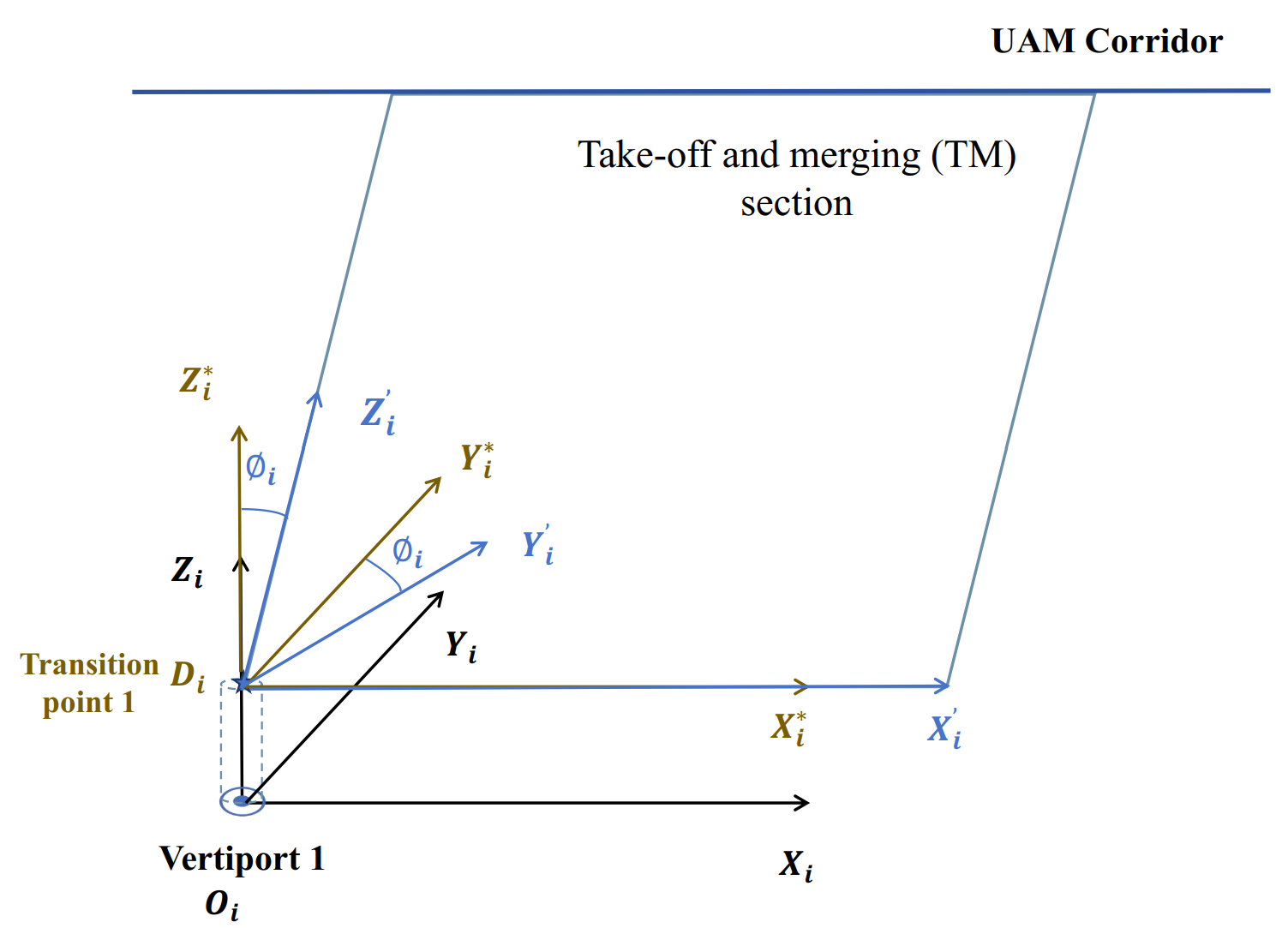}
    \caption{Two-Stage Coordinate Transformation Procedure.}
    \label{fig:12}
\end{figure}
As the translation merely shifts the coordinate origin without altering the dynamics, only the rotation transformation needs to be applied to \eqref{eq:g}. Therefore, \eqref{eq:g} is transformed into the following form:
\begin{equation}
\label{eq:h}
    \dot{\bm{v}'}=\frac{F_i}{m_i} R_{2,i}R_{1,i} e_3-g R_{2,i}e_3
\end{equation}
Expanding the terms in Equation \eqref{eq:h}, we obtain:
\begin{subequations}
\label{eq:23333}
\begin{align}
    \begin{bmatrix}
    \dot{v}_i^{x'} \\
    \dot{v}_i^{y'} \\
    \dot{v}_i^{z'}
    \end{bmatrix}=&\frac{F_i}{m_i}
    \begin{bmatrix}
 1&0&0\\
 0&cos\phi_i&-sin\phi_i\\
 0& sin\phi_i& cos\phi_i
\end{bmatrix}
\begin{bmatrix}
\cos\alpha_i \sin\beta_i \cos\gamma_i + \sin\alpha_i \sin \gamma_i\\
\cos\alpha_i \sin\beta_i \sin\gamma_i - \sin\alpha_i \cos \gamma_i\\
\cos\beta_i \cos\alpha_i
\end{bmatrix}-g \begin{bmatrix}
 1&0&0\\
 0&\cos\phi_i&-\sin\phi_i\\
 0& \sin\phi_i& \cos\phi_i
\end{bmatrix}
\begin{bmatrix}
0\\
0\\
1
\end{bmatrix}\\
=&\frac{F_i}{m_i}
    \begin{bmatrix}
\cos\alpha_i \sin\beta_i \cos\gamma_i + \sin\alpha_i \sin \gamma_i\\
\cos\phi_i(\cos\alpha_i \sin\beta_i \sin\gamma_i - \sin\alpha_i \cos \gamma_i)-\sin\phi_i \cos\beta_i \cos\alpha_i\\
\sin\phi_i(\cos\alpha_i \sin\beta_i \sin\gamma_i - \sin\alpha_i \cos \gamma_i)+\cos\phi_i \cos\beta_i \cos\alpha_i
\end{bmatrix}-g
\begin{bmatrix}
0\\
-sin\phi_i\\
cos\phi_i
\end{bmatrix}
\end{align}
\end{subequations}
Given that the eVTOL's motion is restricted to the $X_iD_iZ_i'$ plane, the lateral acceleration component $\dot{v}_i^{y'}$ is identically zero. Thus, we obtain:
\begin{equation}
\label{eq:A32}
    \frac{F_i}{m_i}[\cos\phi_i(\cos\alpha_i sin\beta_i \sin\gamma_i - \sin\alpha_i \cos \gamma_i)-\sin\phi_i \cos\beta_i \cos\alpha_i]=-g\sin\phi_i
\end{equation}
Then, let us define the vector
\begin{equation}
\label{eq:a1}
    a=R_{2,i}R_{1,i} e_3=\begin{bmatrix}
\cos\alpha_i \sin\beta_i \cos\gamma_i + \sin\alpha_i \sin \gamma_i\\
\cos\phi_i(\cos\alpha_i \sin\beta_i \sin\gamma_i - \sin\alpha_i \cos \gamma_i)-\sin\phi_i \cos\beta_i \cos\alpha_i\\
\sin\phi_i(\cos\alpha_i \sin\beta_i \sin\gamma_i - \sin\alpha_i \cos \gamma_i)+\cos\phi_i cos\beta_i \cos\alpha_i
\end{bmatrix}
\end{equation}
Since $a$ is obtained by applying a sequence of rotation matrices to a unit vector, its norm remains 1. Combine \eqref{eq:A32}, we have:
\begin{subequations}
\begin{align}
&(\cos\alpha_i \sin\beta_i \cos\gamma_i + \sin\alpha_i \sin\gamma_i)^2+ \left[\cos\phi_i(\cos\alpha_i \sin\beta_i \sin\gamma_i - \sin\alpha_i \cos\gamma_i) 
- \sin\phi_i \cos\beta_i \cos\alpha_i \right]^2 \notag \\
&\quad + \left[\sin\phi_i(\cos\alpha_i \sin\beta_i \sin\gamma_i - \sin\alpha_i \cos\gamma_i) 
+\cos\phi_i \cos\beta_i \cos\alpha_i \right]^2 = 1 \\
&(\cos\alpha_i \sin\beta_i \cos\gamma_i + \sin\alpha_i \sin\gamma_i)^2 + \left[\sin\phi_i(\cos\alpha_i \sin\beta_i \sin\gamma_i - \sin\alpha_i \cos\gamma_i) 
+ \cos\phi_i \cos\beta_i \cos\alpha_i \right]^2 \notag \\
&\quad = 1 - \left(\frac{m_i g \sin\phi_i}{F_i}\right)^2
\end{align}
\end{subequations}
Further, since 
\begin{equation}
\begin{aligned}
&\frac{F_i^2}{{F_i}^2-(m_i g \sin\phi_i)^2}
\left[\sin\phi_i(\cos\alpha_i \sin\beta_i \sin\gamma_i - \sin\alpha_i \cos\gamma_i)+ \cos\phi_i \cos\beta_i \cos\alpha_i\right]^2\\
&+\frac{F_i^2}{{F_i}^2-(m_i g \sin\phi_i)^2}
\left(\cos\alpha_i \sin\beta_i \cos\gamma_i + \sin\alpha_i \sin\gamma_i\right)^2  = 1
\end{aligned}
\end{equation}
We can define the tip-path-plane pitch angle $\theta_i$ to represent the angle between the thrust vector and the $ Z'_i$-axis within the $X'_iD_iZ'_i$ plane. And we have
\begin{subequations}
\begin{align}
&\tan\theta_i=\frac{\cos\alpha_i \sin\beta_i \cos\gamma_i + \sin\alpha_i \sin\gamma_i}{\sin\phi_i(\cos\alpha_i \sin\beta_i \sin\gamma_i- \sin\alpha_i \cos\gamma_i)+\cos\phi_i \cos\beta_i \cos\alpha_i}\\
&\sqrt{1 - \left(\frac{m_i g \sin\phi_i}{F_i}\right)^2}\cos\theta_i=
\sin\phi_i(\cos\alpha_i \sin\beta_i \sin\gamma_i - \sin\alpha_i \cos\gamma_i)+ \cos\phi_i \cos\beta_i \cos\alpha_i\label{eq:A35b}\\
&\sqrt{1 - \left(\frac{m_i g \sin\phi_i}{F_i}\right)^2}\sin\theta_i =
\cos\alpha_i \sin\beta_i \cos\gamma_i + \sin\alpha_i \sin\gamma_i\label{eq:A35c}
\end{align}
\end{subequations}
 By \eqref{eq:23333}-\eqref{eq:A32} and \eqref{eq:A35b}-\eqref{eq:A35c}, we can simplify the dynamics of eVTOL $i$ as follows:
\begin{subequations}
\begin{gather}
\dot{x'_i}(t)=v_{i}^{x'} (t)\\
\dot{z'_i}(t)=v_{i}^{z'} (t)\\
\dot{v}_{i}^{x'}(t)=\frac{F_i(t)}{m_i}\sqrt{1 - \left(\frac{m_i g \sin\phi_i}{F_i}\right)^2}\sin(\theta_i(t))\\
\dot{v}_{i}^{z'}(t)=\frac{F_i(t)}{m_i}\sqrt{1 - \left(\frac{m_i g \sin\phi_i}{F_i}\right)^2}\cos(\theta_i(t))-g \cos\phi_i
\end{gather}
\end{subequations}

\section{GPOPS-\uppercase\expandafter{\romannumeral2} parameter settings}

\autoref{tab:03} presents the parameter settings of the adopted GPOPS-\uppercase\expandafter{\romannumeral2} toolbox.

\begin{table}[h]
    \centering
    \caption{Settings of the GPOPS-\uppercase\expandafter{\romannumeral2} toolbox }
    \begin{tabular}{lll}
        \toprule
        Stage & Parameter & Value  \\
        \midrule
        Problem 3 and 4 & setup.nlp.solver & snopt\\  
        Problem 3 and 4 & setup.derivatives.derivativelevel & second\\ 
        Problem 3 and 4 & setup.derivatives.supplier & sparseCD\\
        Problem 3 and 4 & setup.method & RPMintegration\\
        Problem 3 and 4 & setup.mesh.method & hp1\\
        Problem 3 and 4 & setup.mesh.phase.colpoints & 4*ones(1,1)\\
        Problem 3 and 4 & setup.mesh.phase.fraction  & ones(1,1)/1\\
        Problem 3 & setup.mesh.tolerance & $10^{-2}$\\
        Problem 4 & setup.mesh.tolerance & $10^{-6}$\\
        \bottomrule
    \end{tabular}
    \label{tab:03}
\end{table}

\end{appendices}

%别人论文中可参考的段落
%We emphasize that more advanced models for the dynamics of drones can be used, and the methods presented in the ensuing can be applied without loss of generality. This applies equally for including the impact of external conditions, such as wind, on the dynamics of the drones. Here, we have decided to keep the dynamics as simple as possible to keep the mathematics as simple as possible.
%However, among invariant sets whose descriptions have bounded complexity (e.g., sets described using finite degree polynomials), the smallest set may not be one that does not intersect Xu. Not only that, such smallest invariant set may be very difficult to find and may not be unique. Our approach, on the other hand, uses an arbitrary invariant set containing X0 that does not intersect Xu. As such, our method is computationally much easier than the smallest invariant set approach.
%Note that this paper mainly studies the merging scenario in which safe merging gaps are available on the mainline.
\section*{Acknowledgement}

Financial support from the National Natural Science Foundation of China (No. 72071214) and Peng Cheng Laboratory (No. PCL2024Y02) is gratefully acknowledged.

\bibliographystyle{misc/abbrvnat}
\bibliography{UAM_refS}

\end{document}